\documentclass{article}
\usepackage{amsmath}
\usepackage{amssymb}
\usepackage{biblatex}
\usepackage{amsthm}
\usepackage{xcolor}
\usepackage{hyperref}
\usepackage{cleveref}
\usepackage[all]{hypcap}
\usepackage{graphicx}
\usepackage{import}

\usepackage{crossreftools}
\theoremstyle{plain}
\newtheorem{definition}       {Definition}    [section]
\newtheorem{theorem}       {Theorem}    [section]
\newtheorem{conjecture}       {Conjecture}    [section]
\newtheorem{lemma}         {Lemma}      [section]

\newtheorem{proposition}         {Proposition}      [section]
\newtheorem{remark}        {Remark}     [section]

\csname fastrecompileendpreamble\endcsname
\bibliography{main}
\newcommand\ZZ{\mathbb{Z}}
\newcommand\RR{\mathbb{R}}
\begin{document}
\newcommand\CH{\operatorname{CH}}
\newcommand\Int{\operatorname{Int}}
\newcommand\absCH[1]{\lvert\CH(#1)\rvert}

\title{On existence of a compatible triangulation with the double circle order type}
\author{Hong Duc Bui}
\maketitle

\begin{abstract}
	We show that the ``double circle'' order type and some of its generalizations have a compatible triangulation
	with any other order types with the same number of points and number of edges
	on convex hull,
	thus proving another special case of the conjecture in \cite{Aichholzer_2003}.
\end{abstract}

\section{Introduction}

Let $P$ be a set of points on the plane.

\phantomsection\label{def_general_position}
Following definition in \cite{Aichholzer_2003}, we say $P$ is in general position if no three points in $P$ are collinear.
For the rest of this article, we implicitly assume point sets are in general position when convenient.

We say $P$ is in \emph{general$^+$ position} if $P$ is in general position and no three lines through six distinct points in $P$ are concurrent.

Define $\CH(P)$ to be the convex hull of $P$.

\phantomsection\label{def_order_type}
We define the concept of order types: Two point sets $P$ and $Q$ are said to have the same order type if there is a bijection $f \colon P \to Q$ such that $(p_i, p_j, p_k)$ are in counterclockwise order if and only if
$(f(p_i), f(p_j), f(p_k))$ are in counterclockwise order.

Given a point set $P$,
define an \emph{edge} of $P$ to be a line segment with two distinct endpoints in $P$.
A \emph{triangulation} of $P$ is a maximal set of non-intersecting edges of $P$.

Consider two point sets $P = \{ p_1, \dots, p_n \}$ and $Q= \{ q_1, \dots, q_n\}$ for some positive integer $n$.
Let $f\colon P \to Q$ be a bijection.

\phantomsection\label{def_compatible_triangulation}
We say $(T_P, T_Q)$ is a \emph{compatible triangulation} (or joint triangulation) of $(P, Q)$
with respect to the mapping $f$
if $T_P$ is a triangulation of $P$,
$T_Q$ is a triangulation of $Q$,
and whenever the edge $(p_i, p_j)$ is in $T_P$,
the edge $(f(p_i), f(p_j))$ is in $T_Q$, and vice versa.

In order for a compatible triangulation to exist
with respect to some bijection $f$,
it is necessary that $|P| = |Q|$ and $\absCH{P} = \absCH{Q}$.

In \cite{Aichholzer_2003}, its sufficiency was conjectured.
In fact, something stronger was conjectured---even when a bijection of the points on the convex hull are prescribed, it is conjectured that there is still a compatible triangulation as long as the prescribed bijection is a cyclic shift of the convex hulls:
\begin{conjecture}\label{conjecture:main}
	Let $P$ and $Q$ be two arbitrary point sets such that $|P|=|Q|$.
	For any bijection $f_0 \colon \CH(P) \to \CH(Q)$
	preserving the cyclic order of the points on the convex hull,
	there exists some bijection $f \colon P \to Q$ extending $f_0$
	such that there exists a compatible triangulation of $(P, Q)$
	with respect to $f$.
\end{conjecture}

This conjecture was recently mentioned in \cite{o2018open}.

\phantomsection\label{def_cyclic}
Formally, ``$f_0$ preserves the cyclic order of the points on the convex hull'' means the following:
\begin{equation}
	\crefalias{equation}{statement}
	\label{statement_cyclic}
	\text{\parbox{0.75\textwidth}{
			for all $p_1, p_2 \in\CH(P)$, edge $(p_1,p_2)$ is in $\CH(P)$ if and only if edge $(f_0(p_1), f_0(p_2))$ is in $\CH(Q)$.
	}}
\end{equation}
For short, we call such a mapping \emph{cyclic}.

\phantomsection\label{def_universal}
For convenience, we say a point set $P$ is \emph{universal}
if for all point sets $Q$ and cyclic mapping $f_0 \colon \CH(P) \to \CH(Q)$,
the conclusion of \Cref{conjecture:main} holds for $(P, Q, f_0)$.

For example, it is obvious that:
\begin{theorem}
	If $P =\CH(P)$, then $P$ is universal.
\end{theorem}
This is because any triangulation of $P$ can be transported to a triangulation of $Q$. In other words, there is only one convex order type for each number of points $n$.

In fact, \cite{Aichholzer_2003} proved something stronger:
\begin{theorem}
	If $|P-\CH(P)|\leq 3$ or $|P|\leq 8$, then $P$ is universal.
\end{theorem}
The proof of the first statement involves tedious case analysis,
and the proof of the second statement involves using computer brute force through all pairs of point sets.
Other than the cases mentioned, no general families of point sets are known to be universal.

In this article, we prove in \Cref{theorem:result_general_double_circle_3} that
a certain family of point sets is universal,
which gives further evidence towards \Cref{conjecture:main}.

\section{Related Works}

In this section, we list some works on the topic of compatible triangulation.

The problem of compatible triangulation was first considered in \cite{Aichholzer_2003},
which states \Cref{conjecture:main} and make some progress toward it.
\cite{Danciger_2006} studies this problem further.

When the bijection of points is prescribed, \cite{Saalfeld_1987} considers the case
where the points are inside a rectangle and show that there always exist a compatible triangulation
when some additional points, called Steiner points, are added. \cite{diwan2011jointtriangulationssetspoints}
gives an $O(n^3)$ algorithm that conjecturally determines whether there exists a compatible triangulation
with no Steiner points added, and find one if it exists.

\section{Notation}

\begin{table}
	\centering
	\begin{tabular}{cc}
		Term & Note \\ \hline
		$\CH(P)$ & convex hull \\
		general position & see \Cref{def_general_position} \\
		general$^+$ position & see \Cref{def_general_position} \\
		order type & see \Cref{def_order_type} \\
		compatible triangulation & see \Cref{def_compatible_triangulation} \\
		universal point set & see \Cref{def_universal} \\
		cyclic mapping & see \Cref{def_cyclic} \\
		$\mathcal T(P, Q)$, $\mathcal T(P, Q, f_0)$ & see \Cref{def_script_T} \\
		double circle & see \Cref{def_double_circle} \\
		unavoidable spanning cycle & see \Cref{def_unavoidable_cycle} \\
		generalized double circle & see \Cref{def_general_double_circle} \\
	\end{tabular}
	\caption{Some terms defined in this article. If \texttt{hyperref} is enabled, clicking on the links in the table will go directly to the definition.}
\end{table}

Let $P$ be any point set.

We define the double circle as in \cite{Rutschmann_2023}.

\begin{definition}[Double circle]
	\label{def_double_circle}
	Let $P = P_1 P_2 \dots P_n$ be a regular polygon with center $O$.
	For each $1 \leq i \leq n$, let point $A_i$ be the point along the segment
	connecting $O$ and the midpoint of segment $P_i P_{i+1}$,
	and has distance $\varepsilon$ to the segment, where $\varepsilon$ is a sufficiently small positive constant.

	It can be seen that as $\varepsilon \to 0^+$, the order type changes only finitely many times.
	Pick $\varepsilon$ in such a way that for all $0<\varepsilon'<\varepsilon$,
	the order type would be the same if $\varepsilon$ is replaced with $\varepsilon'$.

	Then the double circle is the point set $\{ P_1, \dots, P_n, A_1, \dots, A_n \}$.
\end{definition}

Note that if $i=n$ then $i+1=n+1$, but the point $P_{n+1}$ doesn't exist.
For convenience, we assume indices wraparound when convenient,
so that for example, when $i=n$ then $P_{i+1} = P_1$.

\phantomsection\label{def_unavoidable_cycle}
Also following \cite{Rutschmann_2023},
define an unavoidable spanning cycle as follows.
For a point set $P$ with $|P|=n$,
we say $P_1 P_2 \dots P_n P_1$ is an \emph{unavoidable spanning cycle} of $P$
if $P = \{ P_1, \dots, P_n \}$, and
for every $1 \leq i \leq n$, $P_i P_{i+1}$ belongs to every triangulation of $P$.

Note that the natural cycle on the double circle, $P_1 A_1 P_2 A_2 \dots P_n A_n P_1$, is an unavoidable spanning cycle.
See \Cref{fig:double_circle_illustration} for an illustration.

\begin{figure}
    \centering
    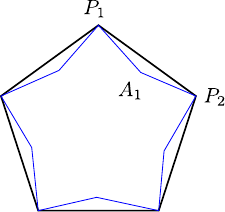
    \caption{Illustration for a double circle. Its unavoidable spanning cycle is colored in blue.}
    \label{fig:double_circle_illustration}
\end{figure}

We generalize \Cref{def_double_circle} as follows. This method of generalization from the double circle
to the generalized double circle is similar to the generalized double chain
defined in \cite{Dumitrescu_2013}.
Also note that 
the set of vertices of an almost-convex polygon,
as defined in \cite{HuNo97},
is a generalized double circle.

\begin{definition}[Generalized double circle]
	\label{def_general_double_circle}
For some integer $n \geq 3$,
for positive integers $c_1, c_2, \dots, c_n$,
define a $(c_1, \dots, c_n)$-double circle to be a point set $P$
such that:
\begin{itemize}
	\item $P$ has $n+\sum_{i=1}^n c_i$ points,
	\item $\CH(P) = n$,
	\item for each $1 \leq i \leq n$, along each edge $P_i P_{i+1}$ of $\CH(P)$, there are $c_i$ points
		$A_{i, 1}, A_{i, 2}, \dots, A_{i, c_i}$ placed along the edge, and slightly offseted towards the center of the polygon,
	\item the polygon $P_i A_{i, 1}, A_{i, 2}, \dots, A_{i, c_i} P_{i+1}$ (in that order) is \emph{convex},
	\item the edges $P_i A_{i, 1}$, $A_{i, c_i} P_{i+1}$, as well as $A_{i, j} A_{i, j+1}$ for every $1 \leq j < c_i$
		forms part of the unavoidable spanning cycle.
\end{itemize}
\end{definition}

See \Cref{fig:illustration_general_double_circle} for an illustration of a $(1, 2, 3)$-generalized double circle
and its unavoidable spanning cycle.

\begin{figure}
    \centering
    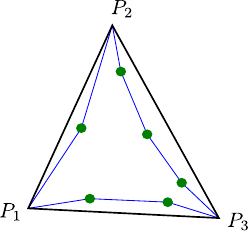
    \caption{A $(1, 2, 3)$-generalized double circle.
		Its unavoidable spanning cycle is also colored in blue.
	}
    \label{fig:illustration_general_double_circle}
\end{figure}

Note that with this definition, the double circle is a $(1,1,\dots, 1)$-double circle.

In the following paragraphs, we introduce the problem of triangulation with Steiner points.
While this notation is not strictly necessary, readers familiar with previous works may find it convenient.

\cite{Aichholzer_2003} also studies the problem of compatible triangulation with Steiner points.
That is, we study the problem of whether there is a compatible triangulation
if $k$ additional points can be added inside the convex hull
of each point set.

Formally,
we say two point sets $P$ and $Q$ have \emph{compatible triangulation with $k$ Steiner points} (per set)
if there exists point sets $S_P$ and $S_Q$ such that:
\begin{itemize}
	\item $|S_P|=|S_Q|= k$,
	\item $P \cap S_P = Q \cap S_P = \varnothing$,
	\item $P \cup S_P$ is in general position;
	\item $Q \cup S_Q$ is in general position;
	\item $\CH(P \cup S_P) = \CH(P)$---in other words, $S_P$ lies inside the convex hull of $P$;
	\item similarly, $\CH(Q \cup S_Q) = \CH(Q)$;
	\item there exists a compatible triangulation of $(P \cup S_P, Q \cup S_Q)$.
\end{itemize}

\phantomsection\label{def_script_T}
For brevity, define $\mathcal T(P, Q)$ to be the minimum $k$ such that point sets $P$ and $Q$ have a compatible triangulation
with $k$ Steiner points;
and for a cyclic mapping $f_0\colon \CH(P) \to \CH(Q)$,
define $\mathcal T(P, Q, f_0)$
to be the minimum $k$ such that point sets $P$ and $Q$ have a compatible triangulation with respect to some bijection $f$ extending $f_0$ with $k$ Steiner points.

We do not require that the bijective function $f\colon P \cup S_P \to Q \cup S_Q$ of the compatible triangulation maps original points to original points and Steiner points to Steiner points. In other words, it is permissible for some $p \in P$ to have $f(P) \in S_Q$, or for $p \in S_P$ to have $f(P) \in Q$.

Clearly, \Cref{conjecture:main} is equivalent to the following:
\begin{quote}
	For all point sets $P$ and $Q$ with $|P|=|Q|$ and $\absCH{P}=\absCH{Q}$,
	for all cyclic mapping $f_0 \colon \CH(P) \to \CH(Q)$,
	then $\mathcal T(P, Q, f_0) = 0$.
\end{quote}

It was proved in \cite[Theorem 3]{Aichholzer_2003} that:
\begin{theorem}
	\label{theorem:existing_result_Steiner_point}
	Let $P$ and $Q$ be two point sets with $h$ points each on the convex hull and $i$ interior points
	(in other words, $\absCH{P}=\absCH{Q}=h$ and $|P|=|Q|=h+i$).
	Then for all cyclic mapping $f_0$, $\mathcal T(P, Q, f_0) \leq i$.
\end{theorem}

We note that we require
$S_P$ to lie inside the convex hull of $P$ (and similarly for $S_Q$ and $Q$), because otherwise the problem would be easy:
\cite[Theorem 3]{Danciger_2006} shows that only two points need to be added.

\section{A Divide-and-conquer Helper Tool}
\label{sec:divide_and_conquer}

The reason why the double circle and related order types are so frequently analyzed in
articles concerning counting the number of triangulation,
such as \cite{Rutschmann_2023,Dumitrescu_2013,Aichholzer2007,erkinger1998struktureigenschaften,HuNo97},
is that it has an unavoidable spanning cycle,
making it easier to count the number of triangulations.
In this article, this property also helps as follows.

Let $P = \{ P_1, \dots, P_n, A_1, \dots, A_n \}$ be the double circle,
and $Q$ be an arbitrary set of points.

Suppose there is a compatible triangulation $(T_P, T_Q)$ with respect to some mapping $f \colon P \to Q$.
For $1 \leq i \leq n$, define $Q_i = f(P_i)$ and $B_i = f(A_i)$.

Then, the unavoidable spanning cycle $P_1 A_1 P_2 A_2 \dots P_n A_n$
gets mapped to $Q_1 B_1 Q_2 B_2 \dots Q_n B_n$,
therefore, the triangulation $T_Q$ contains all of the edges along the polyline $Q_1 B_1 Q_2 B_2 \dots Q_n B_n$.
This data can be interpreted as the association from each point $B_i \in Q \setminus \CH(Q)$
to an edge $Q_i Q_{i+1}$,
such that none of the triangles $\{Q_i B_i Q_{i+1}\}_{i \in \{ 1, \dots, n \}}$ overlap.

Therefore, if we want to find some $(T_P, T_Q, f)$,
we must first be able to find an association.\footnote{
	This is not to say that any method of finding a compatible triangulation must start with finding an association,
	but the problem of finding an association is no harder than finding a compatible triangulation,
	because from a compatible triangulation you can extract out the association by
	seeing where the unavoidable spanning cycle gets mapped to.
}

In this section, we describe a tool that helps us find such an association.
In fact, we are able to prove a more general version:
\begin{theorem}
	\label{convex_subdivision}
	Given a convex polygon $P = P_1 P_2 \dots P_n$,
	points $\{A_1, \dots, A_m \}$ inside polygon $P$,
	such that $P \cup \{ A_1, \dots, A_m \}$ is in general$^+$ position.
	Let $c_i$ be nonnegative integers such that $\sum_{i=1}^n c_i = m$.

	Then there exists a subdivision of the area inside polygon $P$ into
	parts $R_1$, $\dots$, $R_n$
	such that each $R_i$ is convex, has $P_i P_{i+1}$ as an edge,
	has no point in $A$ on its boundary, has exactly $c_i$ points in $A$ in its interior,
	and none of the $R_i$ overlap.
\end{theorem}

See \Cref{fig:convex_subdivision_zero_i} for an illustration of \Cref{convex_subdivision}.
We want each $R_i$ region to have exactly $c_i$ points in $A$.

\begin{figure}
    \centering
    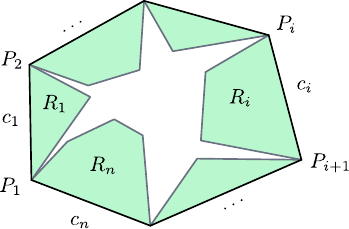
	\caption{Illustration for \Cref{convex_subdivision}.}
    \label{fig:convex_subdivision_zero_i}
\end{figure}

The special case where all $c_i=1$ is equivalent to the association we find.

The core idea is described in the proof of \Cref{prop_triangle_subdivision_3}---%
in fact, \Cref{prop_triangle_subdivision_3} can be seen as almost a special case of \Cref{triangle_subdivision}
with $n = 3$.

\begin{remark}
	\label{remark_extend_weighted_straight_skeleton}
We believe it is possible to extend the proof of \Cref{prop_triangle_subdivision_3}
to work with arbitrary $n$ using weighted straight skeleton,
but we could not handle some degenerate cases.
This idea is explained in more detail in \Cref{subsec_remark_weighted_straight_skeleton}.
\end{remark}

\subsection{Details and Proof of the Algorithm}

\begin{proposition}
	\label{triangle_subdivision}
	Given convex polygon $P = P_1 P_2 \dots P_n$,
	points $\{A_1, \dots, A_m \}$ inside polygon $P$,
	such that $P \cup \{ A_1, \dots, A_m \}$ is in general$^+$ position.
	Let $c_i$ be positive integers such that $\sum_{i=1}^n c_i = m$.
	Then there exists integer $i\in \{ 2, \dots, n-1 \}$ and point $Q$ inside triangle $P_1 P_i P_{i+1}$
	such that within the $m$ points $\{A_1, \dots, A_m \}$,
	there are exactly $c_1+\cdots+c_{i-1}$ points inside polygon $P_1 P_2 \dots P_i Q$,
	$c_i$ points inside triangle $P_i Q P_{i+1}$,
	$c_{i+1}+c_{i+2}+\cdots+c_n$ points inside polygon $P_1 Q P_{i+1} P_{i+2} \dots P_n$.
	Furthermore, $P \cup \{ A_1, \dots, A_m, Q \}$ is in general$^+$ position.
\end{proposition}

We call $P_1$ the \emph{pivot point}.

	Note that the fact that
	$P \cup \{ A_1, \dots, A_m, Q \}$ is in general position
	implies that
	none of $\{ A_1, \dots, A_m \}$ lies on the segments $Q P_1$, $Q P_i$ or $Q P_{i+1}$,
	so they are cleanly split into the three regions separated by these line segments.

See \Cref{fig:polygon_splitting} for an illustration.
Since the existence of a triangulation only depends on the order type, a point set in general position can be perturbed slightly to
make it in general$^+$ position, while not changing the order type.

\begin{figure}
    \centering
    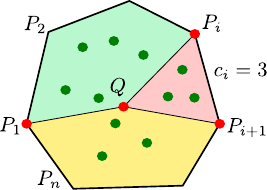
	\caption{Illustration for \Cref{triangle_subdivision}.}
    \label{fig:polygon_splitting}
\end{figure}

\begin{proof}
	For each point $A_j$ in the polygon $P$, define $\sigma(A_j)$ to be the angle $P_2 P_1 A_j$.
	This is between $0$ and $\alpha(P_n)$.
	See \Cref{fig:def_angle_of_point} for an illustration.

	\begin{figure}
		\centering
		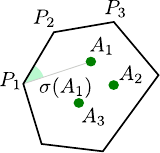
		\caption{Illustration for definition of $\sigma$.}
		\label{fig:def_angle_of_point}
	\end{figure}

	Let $A = \{ A_1, \dots, A_m \}$.
	Sort all points in $A$ in order of increasing $\sigma$.
	Let $j_1, \dots, j_m$ be such that $ \sigma(A_{j_1})< \sigma(A_{j_2})<\cdots< \sigma(A_{j_m})$.
	Note that there cannot be any two different points with equal $\sigma$
	because no three points are collinear.

	Since $\sum_{i=1}^n c_i = m$, we can partition
	$(A_{j_1},A_{j_2},\dots,A_{j_m})$ into contiguous chunks of size $c_i$,
	let $\lambda_i$ and $\rho_i$ be the angle of the leftmost and rightmost point in $i$-th chunk.

	Then $\lambda_i \leq \rho_i$ for each $i$,
	and $\lambda_i = \rho_i \iff c_i = 1$.

	Formally, for $1 \leq i \leq n$,
	\begin{gather*}
		\lambda_i = \sigma\big(A_{j_{(\sum_{k=1}^{i-1} c_i)+1}}\big), \\
		\rho_i = \sigma\big(A_{j_{(\sum_{k=1}^{i} c_i)}}    \big).
	\end{gather*}

	We have $\lambda_2>\sigma(P_2)$ and $\lambda_n \leq \sigma(P_n)$,
	therefore there exists index $i$ such that $\lambda_i>\sigma(P_i)$ and $\lambda_{i+1}\leq\sigma(P_{i+1})$.
	Since no three points are collinear, $\lambda_{i+1}<\sigma(P_{i+1})$,
	so $\rho_i<\sigma(P_{i+1})$ as well.

	Pick this value of $i$.
	Let $m_1$ be the number of points within $A$
	inside polygon $P_1 P_2 \dots P_i$,
	and $m_2$ be the number of points within $A$
	inside polygon $P_1 P_{i+1} P_{i+2} \dots P_n$.

	By the choice of $i$,
	$m_1 \leq \sum_{k=1}^{i-1} c_k$ and $m_2 \leq \sum_{k=i+1}^{n} c_k$.
	Therefore, it suffices if we can find point $Q$ inside triangle $P_1 P_i P_{i+1}$
	such that:
	\begin{itemize}
		\item number of points in $A$ inside triangle $P_1 P_i Q$ is $\sum_{k=1}^{i-1} c_k - m_1$,
		\item number of points in $A$ inside triangle $P_1 Q P_{i+1}$ is $\sum_{k=i+1}^{n} c_k - m_2$,
		\item number of points in $A$ inside triangle $Q P_i P_{i+1}$ is $c_i$.
	\end{itemize}
	The existence of such a point $Q$ will be proven in \Cref{prop_triangle_subdivision_3}.
\end{proof}

The rest of the proof follows.
We also see that this is a special case of
\Cref{triangle_subdivision} when $n = 3$, but slightly more general
because we allow for some $c_i$ being zero,
which is why the statement are almost the same.

\begin{proposition}
	\label{prop_triangle_subdivision_3}
	Given a convex polygon $P = P_1 P_2 \dots P_n$,
	points $\{A_1, \dots, A_m \}$ inside polygon $P$,
	such that $P \cup \{ A_1, \dots, A_m \}$ is in general$^+$ position.
	Let $c_t$ and $c_b$ be nonnegative integers, $c_i$ be a positive integer,
	such that $c_t + c_i + c_b = m$.
	Then there exists a point $Q$ inside triangle $P$ such that within the $m$ points
	$\{ A_1, \dots, A_m \}$,
	there are exactly $c_t$ points inside triangle $P_1 P_i Q$,
	$c_b$ points inside triangle $P_1 Q P_{i+1}$,
	and $c_i$ points inside triangle $Q P_i P_{i+1}$.
\end{proposition}
\begin{proof}
	\newcommand\bracket[1]{\langle #1 \rangle}
	The idea is the following.
	Consider the function $\hat f \colon P \to \ZZ^3$ that maps any point inside triangle $P$ (including the boundary)
	to a triple of integers, defined by $\hat f(Q) = (\bracket{P_1 P_i Q}, \bracket{P_1 Q P_{i+1}}, \bracket{Q P_i P_{i+1}})$,
	where for a triangle $T$, $\bracket{T}$ is the number of points within $\{ A_1, \dots, A_m \}$ inside that triangle.

	Then we wish to find some point $Q$ such that $\hat f(Q) =(c_t, c_b, c_i)$.
	To do that, we find a continuous function $f$ that approximates $\hat f$,
	then apply a result from algebraic topology.

	\renewcommand\bracket[1]{[#1]}

	Let $\varepsilon>0$ be some very small positive integer, which will be selected later.
	For each point $A_j$, let $B(A_j, \varepsilon)$ be the closed ball with radius $\varepsilon$ centered at $A_j$.
	Let $V = \varepsilon^2 \pi$ be the area of such a ball.

	For a geometric object $D$ on the plane, define $|D|$ to be the area of $D$.

	Define a function $f \colon P \to \RR^3$ as follows.
	For each point $Q \in P$,
	\[ f(Q) = (\bracket{P_1 P_i Q}, \bracket{P_1 Q P_{i+1}}, \bracket{Q P_i P_{i+1}}), \]
	where for a triangle $T$, define \[ \bracket{T} = \frac{1}{m \cdot V} \sum_{j=1}^m |B(A_i, \varepsilon) \cap T|. \]

	It is clear that as $\varepsilon \to 0$, $f$ converges to $\hat f$ almost everywhere.
	This is what we mean by $f$ being an approximation of $\hat f$.

	Assume $\varepsilon$ is small enough such that for each $1 \leq j \leq m$,
	$B(A_j, \varepsilon) \subseteq P$.
	Then for each point $Q \in P$, the sum of coordinates of $f(Q)$ is exactly $m$.

	Define $\Delta^2$ to be the closed affine triangle in $\RR^3$ with
	three vertices having coordinate $(1, 0, 0)$, $(0, 1, 0)$ and $(0, 0, 1)$.
	Using the argument above, combined with the fact that each coordinate of $f$ is nonnegative,
	we have the image of $f$ is inside $\Delta^2$.
	We thus restrict the codomain of $f$ to $\Delta^2$.

	Furthermore, we have the following:
	\begin{itemize}
		\item $f(P_1)=(0, 0, m)$,
		\item for each point $Q$ on segment $P_1 P_i$, $f(Q)$ has first coordinate $0$,
		\item $f(P_i)=(0, m, 0)$,
		\item for each point $Q$ on segment $P_i P_{i+1}$, $f(Q)$ has third coordinate $0$,
		\item $f(P_{i+1})=(m, 0, 0)$,
		\item for each point $Q$ on segment $P_{i+1} P_1$, $f(Q)$ has second coordinate $0$.
	\end{itemize}
	Therefore, as $Q$ travels once around the boundary of $P$,
	$f(Q)$ travels once around the boundary of $\Delta^2$.
	In the language of algebraic topology, $f$ maps the boundary of $P$ to the boundary of $\Delta^2$,
	and the induced maps on the first homology group $H_1$ and the fundamental group $\pi_1$ are bijective.

	Therefore, $f \colon P \to \Delta^2$ is surjective,
	using an argument similar to that used in the proof of Brouwer fixed-point theorem.
	This gives us some point $Q$ such that $f(Q) = (c_t, c_b, c_i)$.

	Recall that we want to find $Q$ such that $\hat f(Q) =(c_t, c_b, c_i)$, so we are almost there.

	It remains to prove that for sufficiently small $\varepsilon$ there exists a point $Q$ such that $f(Q) = (c_t, c_b, c_i)$,
	and also $\hat f(Q) = (c_t, c_b, c_i)$.
	We see that it suffices if none of the segments $P_1 Q$, $P_i Q$ and $P_{i+1} Q$
	has any intersection with any of the balls $B(A_j, \varepsilon)$---%
	informally, none of the balls are cut by the segments that separate the three triangles.

	First, we rule out the possibility that $Q$ is inside some ball $B(A_j, \varepsilon)$.
	If this is the case, because no three points are collinear,
	for sufficiently small $\varepsilon$,
	none of the segments $P_1 Q$, $P_i Q$, or $P_{i+1} Q$
	can intersect with any other ball $B(A_k, \varepsilon)$ for $k \neq j$.
	This makes $f$ have non-integral coordinate because $B(A_j, \varepsilon)$ is split into multiple parts,
	so $f(Q) \neq (c_t, c_b, c_i)$, contradiction.

	Then we rule out the possibility that all three segments
	$P_1 Q$, $P_i Q$, and $P_{i+1} Q$
	intersect some balls
	$B(A_{j_1}, \varepsilon)$,
	$B(A_{j_2}, \varepsilon)$,
	and $B(A_{j_3}, \varepsilon)$ simultaneously.
	If all three are the same ball, that is $j_1 = j_2 = j_3$, then point $Q$ must be inside the ball,
	which is already ruled out.
	If two of them are the same ball, let's say $j_1 = j_2$,
	using the same argument as above, because no three points are collinear,
	for sufficiently small $\varepsilon$,
	there cannot be any other $B(A_{j_3}, \varepsilon)$ lying on the remaining segment.
	The remaining case is if all $j_1$, $j_2$ and $j_3$ are distinct,
	then because no three lines are concurrent, for sufficiently small $\varepsilon$ this does not happen either.

	The last case is if there are at most two segments
	$P_1 Q$, $P_i Q$, and $P_{i+1} Q$
	intersecting some ball.
	Again, because no three points are collinear,
	each segment intersects at most one ball for small enough $\varepsilon$.
	Therefore, $f(Q)$ has some nonintegral coordinate, contradiction.
\end{proof}

\begin{remark}
	Without the assumption that no three lines are concurrent,
	the conclusion of \Cref{triangle_subdivision} may not hold.
	See \Cref{fig:three_collinear_lines} for an example.

	\begin{figure}
		\centering
		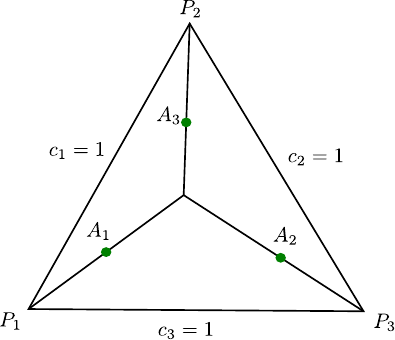
		\caption{Example where the conclusion of \Cref{triangle_subdivision} does not hold.}
		\label{fig:three_collinear_lines}
	\end{figure}

	Also note that we have some freedom in choosing the location of $Q$---%
	which allows us to enforce the general$^+$ position condition.
\end{remark}

We are close to the proof. The following proposition is a slightly weaker version,
where the values $j$ such that $c_j = 0$ are required to be consecutive.

\begin{proposition}
	\label{convex_subdivision_helper}
	Given convex polygon $P$ and point set $A$ as in \Cref{triangle_subdivision}.
	Let $c_1$, $\dots$, $c_n$ be nonnegative integers,
	such that there is some $0 \leq i \leq n$,
	$c_1 = \dots = c_i = 0$,
	$\min(c_{i+1},c_{i+2},\dots, c_n)>0$.
	Then there exists a subdivision of the area inside polygon $P$ into
	parts $R_{i+1}$, $\dots$, $R_n$
	such that each $R_j$ is convex, has $P_j P_{j+1}$ as an edge,
	has no point in $A$ on its boundary, has exactly $c_j$ points in $A$ in its interior,
	and none of the $R_j$ overlap.
\end{proposition}
Note that the remaining part $P \setminus (R_{i+1}\cup \dots \cup R_n)$ may or may not be empty.
See \Cref{fig:convex_subdivision} for an illustration.

\begin{figure}
    \centering
    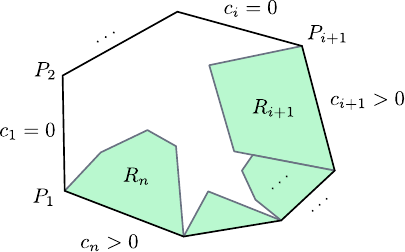
	\caption{Illustration for \Cref{convex_subdivision_helper}.}
    \label{fig:convex_subdivision}
\end{figure}

\begin{proof}
	We induct on $n-i$ i.e. number of nonzero $c_j$ values.

	The simple cases are the following.
	If $n-i=0$ the statement is obvious.
	If $n-i=1$, just let $R_n = P$.
	If $n-i=2$, sort all points counterclockwise around $P_n$,
	draw a ray from it dividing the polygon into two parts, then give one part to $R_{n-1}$
	and the other to $R_n$.

	Assume $n-i>2$. Extend the line $P_1 P_n$ and $P_{i+1}P_{i+2}$, intersecting at $P_\infty$.
	There are three cases that can happen.
	See \Cref{fig:convex_subdivision_2} for an illustration.

	\begin{figure}
		\centering
		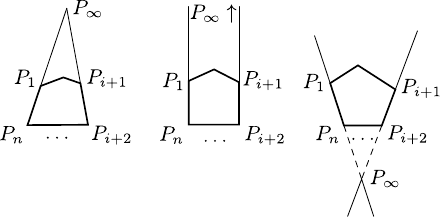
		\caption{Three possible cases for the intersection of line $P_1 P_n$ and $P_{i+1}P_{i+2}$
			in the proof of \Cref{convex_subdivision_helper}.}
		\label{fig:convex_subdivision_2}
	\end{figure}

	These three cases are in fact the same after a projective transformation.
	We will handle the first case, the others are very similar.

	Let polygon $P' = P_\infty P_{i+2} P_{i+3} \dots P_{n-1} P_n$. Then $P'$ is convex.
	Applying \Cref{triangle_subdivision} with the pivot point $P_\infty$
	and corresponding $c_j$ values, we get a division of $P'$ into three parts.
	Label the three parts $\alpha$, $\beta$ and $\gamma$ respectively
	and let $j$ be the edge being chosen for part $\beta$,
	as in \Cref{fig:convex_subdivision_application_step}.

	\begin{figure}
		\centering
		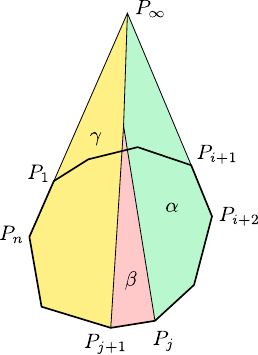
		\caption{Illustration for proof of \Cref{convex_subdivision_helper}.}
		\label{fig:convex_subdivision_application_step}
	\end{figure}

	Then let $R_j = \beta \cap P$,
	and recursively apply the induction hypothesis
	on $\alpha \cap P$ and $\gamma \cap P$,
	which has strictly less number of nonzero $c_j$ values.
\end{proof}

\begin{remark}
	We may not be able to make all the parts triangles.
	See \Cref{fig:counterexample_2} for an example where that cannot be done.
	Because the points $\{A_1, \dots, A_4 \}$ are sufficiently close to the top edge,
	any triangle that contains one of the two segments below and at least two points
	cannot be contained inside the pentagon.

	\begin{figure}
		\centering
		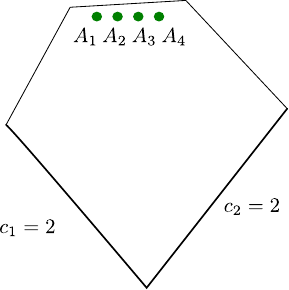
		\caption{Counterexample for \Cref{convex_subdivision} when one assumption is dropped.}
		\label{fig:counterexample_2}
	\end{figure}

	Also we note that the proof actually implies that the union of all $R_i$ equals the whole polygon $P$.
	But this property will be unnecessary from now on.
\end{remark}

Finally, \Cref{convex_subdivision} can be proved as follows.
If $c_j > 0$ for all $1 \leq j \leq n$,
apply \Cref{convex_subdivision_helper} with parameter $i=0$.

In the other case, for each $j$ such that $c_j = 0$,
set $c_j = 1$ and add a point to the set $A$ very near the midpoint of the edge.
It can be proven that if the point added is sufficiently close to the edge $P_j P_{i+1}$,
it must be the only point in $R_j$.
Now that all $c_j > 0$, use the discussion above to find a partition,
and discard the $R_j$ regions corresponding to the edges with $c_j = 0$.

Note that, instead of adding a point to $A$ corresponding to each $j$ such that $c_j = 0$,
we can also extend the two edges $P_{j-1} P_j$ and $P_{j+1} P_{j+2}$ similar to the proof of \Cref{convex_subdivision},
which likely also works,
but then there are more special cases to be handled.

\section{Another Divide-and-conquer Helper Tool}

In this section, we will prove another theorem that also divides a convex polygon into
convex regions, but the regions satisfy a different condition.

\begin{theorem}
	\label{subdivide_2}
	Let $P = P_1 P_2 \dots P_n$ be a convex polygon with no flat angle,
	that is each interior angle is strictly less than $180^\circ$.
	Given additional points $Q_1$, $\dots$, $Q_{n-1}$ satisfying the following:
	\begin{itemize}
		\item Points $Q_i$ are all outside the polygon.
		\item Triangles $P_i Q_i P_{i+1}$ do not intersect pairwise.
		\item All points $Q_i$ are on the same side of line $P_1 P_n$ as the polygon $P$.
	\end{itemize}
	Then we can subdivide the polygon $P$ into a disjoint union of convex possibly-empty polygons
	$P = R_1 \cup \dots \cup R_{n-1}$,
	such that for each $i$, all points inside $R_i$ are visible
	to $Q_i$ through segment $P_i P_{i+1}$.
	(Equivalently, the union of $R_i$ and triangle $P_i Q_i P_{i+1}$ is convex.
	This implies $R_i$ has $P_i P_{i+1}$ as an edge.)
\end{theorem}

See \Cref{fig:subdivide_2} for an illustration when $n = 5$.
There cannot be any point $Q_i$ in the red region.
\begin{figure}
	\centering
	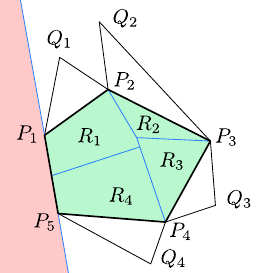
	\caption{Illustration for \Cref{subdivide_2}.}
	\label{fig:subdivide_2}
\end{figure}

\begin{remark}
	The naive algorithm---picking the largest region possible for $R_1$, then $R_2$, etc.
	does not work.  See \Cref{fig:subdivide_hard_case} for an illustration,
	if $R_1$ is chosen to be the largest region possible, then the red region cannot be covered.

	\begin{figure}
		\centering
		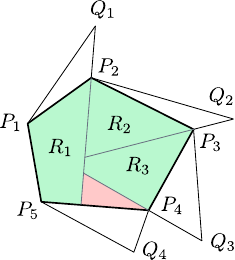
		\caption{Illustration for a case where the naive algorithm for \Cref{subdivide_2} does not work.}
		\label{fig:subdivide_hard_case}
	\end{figure}

	The assumption that there is a line $P_1 P_n$ used as a boundary is indispensable---%
	the theorem does not work when there are $n$ triangles, one outside each edge.
See \Cref{fig:illustration_counterexample_n_triangles_outside} for an illustration.

\begin{figure}
    \centering
    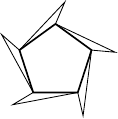
	\caption{Counterexample for \Cref{subdivide_2} when there are $n$ instead of $n-1$ triangles outside.}
    \label{fig:illustration_counterexample_n_triangles_outside}
\end{figure}
\end{remark}

\begin{remark}
	\label{remark_weighted_straight_skeleton}
Our approach is pretty much equivalent to computing the weighted straight skeleton
such as in \cite{Eder2018, Eppstein1998}, with the values of $w_i$ defined below.
However, we are unable to find a proof in the literature
that the weighted straight skeleton gives a convex tessellation.
\end{remark}

\subsection{Proof}

For each $1 \leq i \leq n$, we define a ray $P_i d_i$ as follows.
The idea is that points close to $P_i$ to one side of $d_i$ are visible to $Q_{i-1}$,
points close to $P_i$ close to the other side of $d_i$ are visible to $Q_i$.

See \Cref{fig:splitting_ray} for an illustration.

\begin{figure}
    \centering
    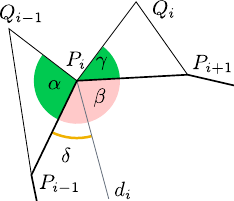
    \caption{Illustration of ray $d_i$ and the angles.}
    \label{fig:splitting_ray}
\end{figure}

For $i = 1$, ray $P_1 d_1$ is ray $P_1 P_n$.
For $i = n$, ray $P_n d_n$ is ray $P_n P_1$.
Otherwise, we perform the following procedure.

Let $\alpha$ be angle $\angle Q_{i-1} P_i P_{i-1}$,
$\beta$ be $\angle P_{i-1} P_i P_{i-1}$,
$\gamma$ be $\angle P_{i+1} P_i Q_i$.

Because triangle $P_{i-1} Q_{i-1} P_i$ and
$P_i Q_i P_{i+1}$ has no overlap,
$\alpha+\beta+\gamma \leq 360^\circ$.
Also we have $\max(\alpha, \beta, \gamma)<180^\circ$.

We want to find $0<\delta<\beta$ (which is the angle between $P_i P_{i-1}$ and $P_i d_i$)
such that $\max(\alpha + \delta, \gamma+\beta-\delta) < 180^\circ$.

To do that, picking any $\delta$ such that $\max(0, \gamma+\beta-180^\circ) < \delta < \min(\beta, 180^\circ-\alpha)$ suffices.
The assumptions easily implies the left side is less than the right side.

After we have constructed $d_i$, define $r_i = \frac{\sin(\beta-\delta)}{\sin \delta}$.
The meaning of this quantity is the following:
define $d_i(A)$ to be the distance from point $A$ to line $P_i P_{i+1}$,
for all point $A$ in ray $P_i d_i$,
then $r_i = \frac{d_i(A)}{d_{i-1}(A)}$.

For each $1 \leq i \leq n-1$, define
$w_i = \prod_{j=1}^{i-1} r_i$. In particular $w_1 = 1$.
Define $d'_i(A) = \frac{d_i(A)}{w_i}$,
then for all point $A$ in ray $P_i d_i$, $d'_i(A) = d'_{i-1}(A)$.

For each $1 \leq i \leq n-1$, define
\[
	R_i = \{ A \in P \mid
		d'_i(A) < d'_j(A) \text{ for all $1 \leq j \leq n-1$, $j \neq i$}
	\}.
\]

\begin{lemma}
	The closure of $R_i$ is a convex polygon.
\end{lemma}
\begin{proof}
	For each $1 \leq i < j \leq n-1$,
	there is an unique line $l_{i,j}$ that divides
	the points in $P$ into two parts,
	one has $d'_i(A) < d'_j(A)$,
	the other is the opposite, $d'_i(A) > d'_j(A)$.

	As such, $R_i$ is the intersection of $P$ and $n-2$ half planes,
	so it is convex.
\end{proof}

By our careful definition of $r_i$ and $w_i$, we have the following:
\begin{lemma}
	\label{prop_seed_for_ri}
	There is a small region around edge $P_i P_{i+1}$, and inside angle $\angle d_{i-1} P_{i-1} P_i$
	and angle $\angle P_{i-1} P_i d_i$ that all belongs to $R_i$.
	Furthermore, small parts near vertex $P_i$
	in angle $\angle P_{i-1} P_i d_i$
	or angle $\angle d_{i+1} P_{i+1} P_{i+2}$ does not belong to $R_i$.
\end{lemma}

See \Cref{fig:seed_for_ri} for an illustration.

\begin{figure}
    \centering
    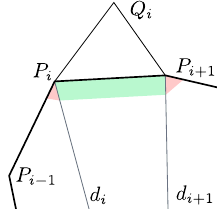
	\caption{Illustration for \Cref{prop_seed_for_ri}. The green strip belongs to $R_i$
		and the red strips do not belong to $R_i$
	if their widths are set small enough.}
    \label{fig:seed_for_ri}
\end{figure}

\begin{proof}
	Since all weights $w_i$ are nonzero, edges far apart cannot interfere with $R_i$ close to
	edge $P_i P_{i+1}$. Only edge $P_{i-1} P_i$ and $P_{i+1} P_{i+2}$ can interfere,
	and their behavior is well-understood.
\end{proof}

As such, because $R_i$ is convex, we have $R_i \cup P_i Q_i P_{i+1}$ is convex as needed.
(Note that the fact that the red strips in \Cref{fig:seed_for_ri} \emph{does not} belong to $R_i$
is critical to prove that all points in $R_i$ can be seen from $Q_i$.)

\subsection{Remark on Weighted Straight Skeleton}
\label{subsec_remark_weighted_straight_skeleton}

As mentioned in \Cref{remark_weighted_straight_skeleton}, the boundaries of $R_i$ forms a weighted straight skeleton.
As we alluded to in \Cref{remark_extend_weighted_straight_skeleton},
we may potentially generalize the argument in
\Cref{prop_triangle_subdivision_3} to work for all $n$.
Formally, we do the following.

Let $P = P_1 \dots P_n$ be a convex polygon.
For each $1 \leq i \leq n$, define $d_i (A)$ to be the distance from point $A$ to line $P_i P_{i+1}$.
For any non-negative reals $w_1, \dots, w_n$, such that not all of them are zero, define
$d'_i (A) = \frac{d_i (A)}{w_i}$; when $w_i = 0$, define $d'_i (A) = +\infty$.
Then for each $1 \leq i \leq n$, define
\begin{equation}
	\label{equation_region_of_weighted_skeleton}
	R_i = \{ A \in P \mid d'_i (A) < d'_j (A) \text{ for all }1 \leq j \leq n, j \neq i \}.
\end{equation}
This is very similar to the definition above, except that we have a region for each of the $n$ edges,
instead of just $n-1$ edges.
Note that if each of $w_i$ is scaled up by a constant, $R_i$ remains unchanged,
therefore without loss of generality, we may assume $\sum_{i=1}^n w_i = 1$.

Now we proceed similar to the proof of \Cref{prop_triangle_subdivision_3}.
Pick a small $\varepsilon>0$.
Define $\Delta^{n-1} = \{ (x_1, \dots, x_n) \in \RR_{\geq 0}^n \mid \sum_{i=1}^n x_i = 1 \}$.
We define a continuous function $f \colon \Delta^{n-1}\to \Delta^{n-1}$ by:
the $i$-th coordinate of $f(w_1, \dots, w_n)$ is
$\frac{1}{m \cdot V} \sum_{i=1}^m |B(A_i, \varepsilon) \cap R_i|$,
where the region $R_i$ are defined 
in \Cref{equation_region_of_weighted_skeleton}
using $w_1, \dots, w_n$ as the weights.
Then again, $f$ maps boundaries to boundaries, thus it is surjective.

Informally, what we have done is the following:
inflate each point $A_j$ to a ball of radius $\varepsilon \geq 0$,
construct the weighted straight skeletons at weights $(w_1, \dots, w_n)$
to divide the polygon $P$ into regions $\{ R_i\}_{i \in \{ 1, \dots, n \}}$,
then count the fraction of the areas of the balls in each region.
As the function $f$ is surjective, there exists a choice of weights
such that region $R_i$ has a fraction $\frac{c_i}{m}$ of the area.

Unfortunately, we are not able to prove that if the points
$P_i$ and $A_j$ are in general$^+$ position, then for sufficiently small $\varepsilon$,
none of the balls are cut.

\section{On Compatible Triangulation of Polygons}

Triangulation of polygons is a well-studied problem,
with \cite{Garey1978,Chazelle1991} giving fast algorithms to compute a triangulation,
the best known algorithm is linear-time.

\cite{Lubiw_2020,Gupta_1997,Aronov_1993} study compatible triangulations
of polygons and polygonal regions, which can be seen as a more constrained version of
point set triangulation where some edges are prescribed.

Formally, let $P = P_1 P_2 \dots P_n$ be a non-self-intersecting polygon.
A triangulation $T_P$ of polygon $P$ is a maximal set of edges all contained inside $P$.

Note that the notation in this section is slightly different from the other sections.

In our work, in order to prove the main theorem, we need a few claims on compatible triangulations of polygons.

\begin{lemma}
	\label{lemma:vertex_ordering}
	Let $P$ be a polygon.
	Assume $i$, $j$, $k$ and $l$ are distinct indices
	such that $i<j$ and diagonals $P_i P_j$ and $P_k P_l$ are both contained inside $P$.
	Then the two diagonals intersect
	if and only if exactly one of $k$ and $l$ belongs to the set $\{ i+1, i+2, \dots, j-1 \}$.
\end{lemma}
\begin{proof}
	Since diagonal $P_i P_j$ is contained inside $P$,
	it divides polygon $P$ into two polygons
	$P_1 P_2 \dots P_{i-1} P_i P_j P_{j+1} \dots P_{n-1} P_n$
	and
	$P_i P_{i+1} \dots P_{j-1} P_j$.

	If both $k$ and $l$ belongs to $\{ i+1, \dots, j-1 \}$,
	then both $P_k$ and $P_l$ belongs to the second polygon,
	$P_i P_{i+1} \dots P_{j-1} P_j$.
	Therefore the whole segment $P_k P_l$ must also belong here---%
	otherwise there would be some part of it belong to the other polygon,
	then we can find two intersections of $P_k P_l$ with the segment $P_i P_j$
	(which divides the two polygons), contradiction.

	Similar argument applies when neither of $k$ nor $l$ belongs to $\{ i+1, \dots, j-1 \}$.

	Conversely, if the two diagonals have no intersection, then the straight segment from $P_k$ to $P_l$
	does not go through $P_i P_j$, therefore $P_k$ and $P_l$ belongs to the same polygon.
\end{proof}

\begin{proposition}
	\label{prop:porting_one_triangulation}
	Let $P = P_1 \dots P_n$ and $Q = Q_1 \dots Q_n$ be polygons.
	Let $T_Q$ be a given triangulation of $Q$.
	Assume for every $(Q_i, Q_j) \in T_Q$ that is not an edge of $Q$,
	$(P_i, P_j)$ is a diagonal of $P$ (that is, the whole segment is inside $P$).
	Then the set of edges
	\[ T_P = \{ (P_i, P_j) \mid (Q_i, Q_j) \in T_Q \} \]
	is a triangulation of $P$.
\end{proposition}

\begin{proof}
	We prove $T_P$ has no nontrivial intersection.
	This is because no two diagonals in $T_Q$ intersects nontrivially,
	and apply \Cref{lemma:vertex_ordering} we see that the two corresponding diagonals in $T_P$
	does not intersect nontrivially either.
	
	Because $T_Q$ has exactly $n-3$ segments strictly inside polygon $Q$,
	$T_P$ has exactly $n-3$ segments strictly inside polygon $P$.

	Therefore $T_P$ is a triangulation.
\end{proof}

\begin{proposition}
	\label{prop:porting_triangulation}
	Let $P = P_1 \dots P_n$ and $Q = Q_1 \dots Q_n$ be polygons.
	Assume for every $i$ and $j$ such that diagonal $Q_i Q_j$ lies completely inside the polygon $Q$,
	then $P_i P_j$ also lies inside the polygon $P$.
	Then given any triangulation $T_Q$ of $Q$, the set of edges
	\[ T_P = \{ (P_i, P_j) \mid (Q_i, Q_j) \in T_Q \} \]
	is a triangulation of $P$.
\end{proposition}

In the notation of \cite{Ghosh_1997},
the precondition of \Cref{prop:porting_triangulation} is that
the mapping $(Q_i, Q_j) \mapsto (P_i, P_j)$
embeds the visibility graph of the simple polygon $Q$ into the visibility graph of the simple polygon $P$.

\begin{proof}
	From the assumption, we get that all the diagonals of $Q$ in $T_Q$ get mapped
	to a diagonal of $P$.
	Apply \Cref{prop:porting_one_triangulation}, we are done.
\end{proof}

\begin{remark}
Applying the proposition with $P$ being a convex polygon and $Q$ being an arbitrary polygon,
we recover the intuitive fact that a convex polygon has the largest number of triangulations
within the polygons with a fixed number of vertices.
\end{remark}

\section{Main Result}

\begin{theorem}
	\label{theorem:universality_double_circle}
	The double circle is universal.
\end{theorem}

\begin{proof}
	Let $\{ P_1, \dots, P_n, A_1, \dots, A_n \}$ be a double circle,
	and $Q = \{ Q_1, \dots, Q_n, \allowbreak B_1, \dots, B_n \}$ be an arbitrary set of points
	with $\{ Q_1, \dots, Q_n \}$ on the convex hull and the rest in the interior.
	(The notation is the same as in \Cref{sec:divide_and_conquer}.)
	For convenience, assume the prescribed mapping $f_0$ maps $P_i$ to $Q_i$ for each $i$.

	Then, we apply \Cref{convex_subdivision} on $Q$
	with $c_1 = \dots = c_n = 1$,
	getting parts $R_1$, $\dots$, $R_n$.

	For convenience, we relabel the points $B_i$ such that $B_i$ is in part $R_i$.
	Then define the function $f$ extending $f_0$ by $f(A_i) = B_i$.

	Because each part $R_i$ is convex, each triangle $Q_i B_i Q_{i+1}$ is contained in $R_i$,
	so none of the triangles $Q_i B_i Q_{i+1}$ overlap for different values of $i$.

	Find any triangulation of the polygon $Q_1 B_1 Q_2 B_2 \dots Q_n B_n$.
	Then port the triangulation to $P_1 A_1 P_2 A_2 \dots P_n A_n$.
	It must be a valid triangulation.

	Formally, let $T_Q$ be any triangulation of $Q$ containing the edges
	$\{ Q_i Q_{i+1}, \allowbreak Q_i B_i, \allowbreak B_i Q_{i+1} \}$ for each $1 \leq i \leq n$.
	Since we have shown above that the prescribed edges have no intersection,
	such a triangulation exists.
	Then apply \Cref{prop:porting_triangulation}.

	Finally, extend the triangulation of the polygons to a triangulation of the whole point set
	by adding the edges on the convex hull.
\end{proof}

The double circle is the special case of the generalized double circle where $c_i = 1$ for all $i$.
When $c_i > 1$, it is not so easy to apply \Cref{prop:porting_triangulation} to port from a triangulation of the inner region
of $Q$ to $P$---for example, in the left panel of \Cref{fig:counterexample_no_inner_triangulation},
if the red polyline is to be mapped to the unavoidable spanning cycle of a $(2, 2, 2)$-generalized double circle,
then there is no compatible triangulation.

However, we still can prove it.
\begin{theorem}
	\label{theorem:result_general_double_circle_3}
	Let $c_1$, $\dots$, $c_n$ be nonnegative integers.
	Let $P$ be the $(c_1, \dots, c_n)$-generalized double circle.
	Then $P$ is universal.
\end{theorem}
\begin{proof}
	Let 
	the points of $P$
	be 
	\begin{align*}
		P & = \{ P_1, \dots, P_n, \\
		  & \qquad A_{1, 1}, \dots, A_{1, c_1}, \\
		  & \qquad A_{2, 1}, \dots, A_{2, c_2}, \\
		  & \qquad \vdots \\
		  & \qquad A_{n, 1}, \dots, A_{n, c_n} \}
	\end{align*}
	with notation as in \Cref{def_general_double_circle}.

	Given any point set $Q = \{ Q_1, \dots, Q_n, B_1, \dots, B_m \}$
	in general position such that $|P|=|Q|$, $\absCH{P}=\absCH{Q}$.

	Since the existence of a triangulation only depends on the order type of $Q$,
	we may perturb $Q$ slightly so that no three lines are concurrent.

	Then apply \Cref{convex_subdivision} to divide the area in polygon $\CH(Q)$ into
	disjoint polygons $R_1$, $\dots$, $R_n$ (whose union not necessarily cover $Q$).

	For each $i$, let polyline $L_i$ be edge $Q_i Q_{i+1}$ if $c_i = 0$,
	otherwise polyline $L_i$ is the convex hull of
	$Q \cap R_i$ (including the boundary of $R_i$, therefore $Q_i$ and $Q_{i+1}$ is included)
	minus segment $Q_i Q_{i+1}$.

	Let $L$ be the closed polyline formed by concatenating $L_1 L_2 \dots L_n$.
	Since the polygons $R_i$ are disjoint, $L$ is not self-intersecting,
	thus forms a simple polygon.

	Triangulate the region inside $L$ (this is possible because of \cite{Meisters1975}).

	Next, for each $i$ such that not all of points in $B \cap R_i$ is on the polyline $L_i$, sequentially perform the following procedure:
	\begin{itemize}
		\item Let $L_i = T_1 T_2 \dots T_p$,
			where $T_1 = Q_i$ and $T_p = Q_{i+1}$.
		\item Note that for each edge $T_i T_{i+1}$, there is a triangle outside it in the triangulation computed above.
		\item Divide polygon $T_1 T_2 \dots T_p$ into convex polygons using \Cref{subdivide_2}.
		\item For each such polygon, sort the points in the region by its angle with respect to its tip,
			then modify the polyline $L_i$ to incorporate these points on the chain.
	\end{itemize}
	See \Cref{fig:augment_polyline} for an illustration.

\begin{figure}
    \centering
    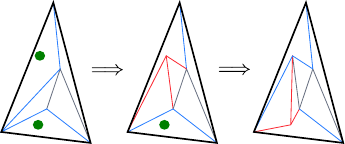
	\caption{Illustration for proof of \Cref{theorem:result_general_double_circle_3}.}
    \label{fig:augment_polyline}
\end{figure}

Finally, map the concatenation of modified chains $L_i$ to the unavoidable spanning cycle of $P$,
and triangulate the area on the border of $Q$ arbitrarily then port it to $P$.
\end{proof}

For extra clarity, the following is a more formal explanation of the algorithm.
Fix an index $i$, let $L_i = T_1 T_2 \dots T_p$ be as above.
See \Cref{fig:triangulate_inside_a_ri} for an illustration.

\begin{figure}
    \centering
    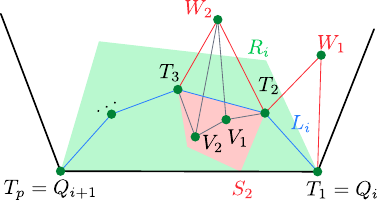
	\caption{Illustration for proof of \Cref{theorem:result_general_double_circle_3}.}
    \label{fig:triangulate_inside_a_ri}
\end{figure}

Because the region inside $L$ is already triangulated, for each $1 \leq j < p$,
we can find a triangle $W_j T_j T_{j+1}$ inside $L$ with $T_j T_{j+1}$ as a segment.

Because the triangles $W_j T_j T_{j+1}$ belong to the given triangulation, they do not intersect pairwise.
Furthermore, $W_j$ are all on the same side of line $Q_i Q_{i+1}$ as polygon $T_1 T_2 \dots T_p$,
because polygon $Q = Q_1 Q_2 \dots Q_n$ are convex
and all points $W_j$ belongs to polygon $Q$.

Therefore \Cref{subdivide_2} can be applied on polygon $T_1 T_2 \dots T_p$ and
additional points $W_1$, $W_2$, $\dots$, $W_{p-1}$.
As a result, we get $T_1 T_2 \dots T_p = S_1 \cup S_2 \cup \dots \cup S_{p-1}$
for convex regions $S_1$, $\dots$, $S_{p-1}$.

For each $1 \leq j \leq p-1$,
let the points inside $B \cap S_j$ be $\{ V_1, \dots, V_q \}$,
ordered clockwise around point $W_j$.
For example, in \Cref{fig:triangulate_inside_a_ri},
seen from point $W_2$, the counterclockwise ordering is $(T_2, V_1, V_2, T_3)$.

If $q = 0$, nothing needs to be done with this value of $j$.
Otherwise, we modify polyline $L_i$ as follows.
Previously, it has a single segment $T_j T_{j+1}$.
Cut that line away, and add the polyline $T_j V_1 V_2 \dots V_q T_{j+1}$ in its place.

We also modify the triangulation. Previously there was a triangle
$W_j T_j T_{j+1}$.
We remove that triangle, and add triangles $W_j T_j V_1$, $W_j V_1 V_2$,
$\dots$, $W_j V_{q-1} V_q$, $W_j V_q T_{j+1}$.
Because the union of region $S_j$ and triangle $W_j T_j T_{j+1}$ is convex
and the points $V_k$ are ordered counterclockwise around point $W_j$,
all of the triangles constructed above belongs to the union of region $S_j$ and triangle
$W_j T_j T_{j+1}$; and they does not intersect pairwise.

Therefore, we can perform the procedure independently for each value $1 \leq j < p$,
and at the end, the triangulation is still a valid triangulation inside the polyline
$L = L_1 L_2 \dots L_n$.

\section{Discussion}

\subsection{Important Special Cases with Projective Geometry}

We see that using projective geometry simplifies the proof of \Cref{convex_subdivision}.

There is a special case of the problem:
when the interior points are quite close to each other.
Using projective geometry, this can be seen as equivalent to having the points of the polygon being at infinity.

In trying to resolve the main conjecture, it may be useful to attempt this special case first.

In our attempts at this special case, we observe that ideas that work in this case can often be generalized to the general case.

\subsection{Matching Problem and Acyclicity of State Graph}

In the proof of \Cref{theorem:universality_double_circle}, the difficult part is to produce a non-intersecting matching of points $B_i$ and edges of $\CH(Q)$.
In some sense this is similar to a well known problem:
given a set $R$ of $n$ red points and $B$ of $n$ blue points in general position on the plane,
find a matching of the red points to the blue points such that no segments connecting the matched
pairs intersect.
As mentioned in 
the introduction of 
\cite{Aloupis2010},
the matching with minimum total length of the segments is non-crossing.\footnote{%
In fact, \cite{Aloupis2010} also studies the case of matching points to the edges
of an enclosing convex polygon in section 4, however the case studied assumes the
bijection between points and edges is prescribed (that is, the set of points $P$ and the set of edges $T$ is ordered),
thus bypassing the main difficulty of our problem. Furthermore, a matching is defined in 
\cite{Aloupis2010} as a line segment with endpoint anywhere inside each object;
while in this article, we need a triangle connecting a point and a line segment.
}

There is also the following simple algorithm to solve the red-blue point matching above:
start with any matching, then while there is an intersection let's say between segment $R_1 B_1$
and $R_2 B_2$, then change the matching to match $R_1 B_2$ and $R_2 B_1$ instead.
Since the sum of length of all segments monotonically decrease and there are finitely many matchings,
this procedure will always terminate.

In the case $c_1 = \dots = c_n = 1$,
the same heuristic algorithm to find a matching:
start with any matching from points to edges, then while there is an intersection
let's say between triangles $Q_1 B_1 Q_2$ and $Q_3 B_3 Q_4$,
swap the matching (to get triangles $Q_1 B_3 Q_2$ and $Q_3 B_1 Q_4$).

\begin{conjecture}
	The algorithm described above always terminate.
\end{conjecture}

Formally: we represent a matching
where segment $Q_i Q_{i+1}$ is matched to point $B_{\pi_i}$
with a permutation $\pi = (\pi_1, \dots, \pi_n)$ of $(1, \dots, n)$.
Then, for each point set $Q$ with $|Q|=2n$ and $\absCH{Q}=n$, define the directed graph $G$ as follows:
\begin{itemize}
	\item The vertices are all $n!$ permutations of $(1, \dots, n)$;
	\item For two permutations $\pi$ and $\rho$, there is an edge from $\pi$ to $\rho$ if and only if
		there are two indices $i\neq j$, $\pi_i = \rho_j$, $\pi_j = \rho_i$, $\pi_k = \rho_k$
		for all $k \in \{ 1, \dots, n \} \setminus \{ i, j \}$;
		and triangles $Q_i B_{\pi_i} Q_{i+1}$ and $Q_j B_{\pi_j} Q_{j+1}$ intersect.
		(It can then be shown in this case then
		triangles $Q_i B_{\rho_i} Q_{i+1}$ and $Q_j B_{\rho_j} Q_{j+1}$ does not intersect.)
\end{itemize}
Then the conjecture states the graph $G$ is acyclic for all order types $Q$ in general position.

Using the datasets provided by \cite{Aichholzer_2001}, we have verified that the conjecture holds for all $n \leq 5$
(thus $2n \leq 10$). For $2n = 8$, $1468$ out of $3315$ order types have exactly half of the points on the convex hull;
and for $2n = 10$, $2628738$ out of $14309547$ order types have exactly half of the points on the convex hull.
Since a significant fraction of all order types with an even number of points (44\% and 18\% for $2n=8$ and $2n=10$ respectively) have exactly half of its points on the convex hull,
we expect brute force to be infeasible for $n \geq 7$.

\subsection{Open Problems}

Apart from the conjectures posed above and the obvious problem of proving the compatible triangulation conjecture,
a possibly more approachable step at the moment is to generalize the proof to
polygon where each side is a chain, where a chain is defined in \cite{Rutschmann_2023}.

It is probably difficult to prove the case of almost-convex polygon as defined in \cite{bacher2010},
since every order type can be represented as an almost-convex polygon.

It would also be interesting to see whether it is possible to have an algorithm that computes
the compatible triangulation in sub-quadratic time.

\printbibliography

\clearpage

\appendix

The appendix contains unpolished contents that proves some special cases and is less elegant.
There are also some attempted directions that does not work out,
of course they are made redundant by the proofs in the main sections.

\section{Triangulating the Area between \texorpdfstring{$2$}{2} Polylines}

In this section, we will describe a tool that will be useful in an old proof of \Cref{theorem:result_general_double_circle_2}.

We use the following notation in this section:
\begin{itemize}
	\item $a$, $b$ and $c$ are perfectly vertical lines,
		such that line $a$ is to the left of line $b$, and line $b$ is to the left of line $c$.
	\item $P$ is a nonempty point set between two lines $a$ and $b$ with $n$ points.
	\item $Q$ is a nonempty point set between two lines $b$ and $c$ with $m$ points.

		We allow for the case where some point of $P$ is on $a$, or some point of $P$ is on $c$.
		But no point of $P$ or $Q$ can be on $b$.
	\item $X$ is a point that is ``infinitely far'' above the figure, and $Y$ is a point that is
		``infinitely far'' below the figure.

		As such, when we say e.g. ``segment $AX$'' for a finite point $A$, we mean
		the ray starting from $A$ and travel upward. Similarly ``line $AX$'' is the line
		passing through $A$ and is perfectly vertical.

		Similarly, ``three points $A$, $B$ and $X$ are collinear'' means ``line $AX$ passes through $B$''.

	\item $P \cup Q \cup \{ X, Y \}$ are in general position.

		One implication of this is that there can be at most one point of $P$ on $a$ (if there are two,
		these two would be collinear with $X$).
\end{itemize}

Since $P$ is in the region between lines $a$ and $b$,
and $Q$ is in the region between lines $b$ and $c$,
a line passing through any $P_i \in P$ and $Q_i \in Q$
cannot be perfectly vertical,
therefore it splits the plane into two regions, which we can call ``above'' and ``below'' the line.

\begin{proposition}
	\label{prop:upper_cap_prop}
	There exists points $P_1 \in P$ and $Q_1 \in Q$ such that
	all the points in $P \cup Q$ except $P_1$ or $Q_1$
	are below the line $P_1 Q_1$. Furthermore, the choice is unique.
\end{proposition}
\begin{proof}
	Consider the convex hull of $P \cup Q$, it intersects line $b$ at two points.
	Pick the two endpoints of the edge that intersects line $b$ at a higher point,
	we're done.
	It is not difficult to see the choice is unique either.
\end{proof}

We call the line segment $P_1 Q_1$ the \emph{upper cap} of $P \cup Q$.
The \emph{lower cap} can be defined analogously.

\begin{theorem}
	\label{theorem:triangulate_between_two_polylines}
	There exists an ordering $P_1, \dots, P_n$ of points in $P$,
	and an ordering $Q_1, \dots, Q_m$ of points in $Q$,
	such that $P_1 Q_1$ is the upper cap of $P \cup Q$,
	the polylines $X P_1 P_2 \dots P_n Y$ and $X Q_1 Q_2 \dots Q_m Y$ are not self-intersecting,
	and the region between these two polylines can be triangulated using only edges
	that intersects line $b$.

	Alternatively, if the condition
	``$P_1 Q_1$ is the upper cap of $P \cup Q$''
	is replaced with
	``$P_n Q_m$ is the lower cap of $P \cup Q$'',
	the statement still holds.
\end{theorem}

See \Cref{fig:counterexample_both_cap} for an illustration.
Also note that there may not be any ordering such that
both $P_1 Q_1$ is the upper cap and $P_n Q_m$ is the lower cap of $P \cup Q$.
The situation in the \Cref{fig:counterexample_both_cap} is such a case.

\begin{figure}
	\centering
	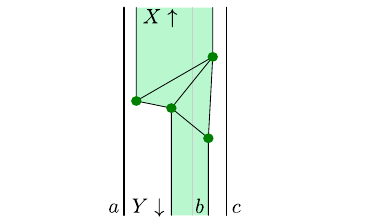
	\caption{Example case where not both caps can be endpoints.}
	\label{fig:counterexample_both_cap}
\end{figure}

\begin{proof}
	One possibility is to induct on the first line.
	But adapting the proof technique of \Cref{theorem:result_general_double_circle_3}
	is evidently better.
\end{proof}

Applying a projective transformation, we get the following.

Let $X$ be a point on the plane, with three rays $Xa$, $Xb$ and $Xc$
as in \Cref{fig:two_polyline_finite_case}.

\begin{figure}
    \centering
    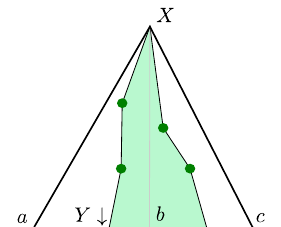
    \caption{Illustration for \Cref{theorem:triangulate_between_two_polylines_finite}.}
    \label{fig:two_polyline_finite_case}
\end{figure}

Let $P$ be a set of points inside $aXb$, and $Q$ be a set of points inside $bXc$.
Let $Y$ be a formal point defined such that for an (actual) point $A$,
``segment $AY$'' is ``the ray opposite ray $AX$''.
\begin{theorem}
	\label{theorem:triangulate_between_two_polylines_finite}
	With notation changed as above,
	the result stated in \Cref{theorem:triangulate_between_two_polylines} still holds.
\end{theorem}
We call $X$ the \emph{pivot point}, and $Xb$ the \emph{separating ray}.
\begin{proof}
	Apply a projective transformation that sends $X$ to a point at infinity,
	apply \Cref{theorem:triangulate_between_two_polylines}, then transform back.
\end{proof}

\section{Triangulating the Area between \texorpdfstring{$3$}{3} Polylines}

We generalize the tool described in the previous section.

\begin{theorem}
	\label{theorem:triangulate_between_three_polylines}
	Let $XYZ$ be a triangle, point $F \in XYZ$,
	$A$ is a set of points inside $XYZ$,
	$\{ X, Y, Z, F \} \cup A$ is in general position.
	
	Define $P_1 = X$, $P_2 = Y$, $P_3 = Z$.

	Let $c_i$ be the number of points in the intersection of point set $A$
	and triangle $F P_i P_{i+1}$.
	Assume $c_i$ are all positive integers.

	Then there exists a permutation of points in $A$
	\[ (
		A_{1, 1}, \dots, A_{1, c_1},
		A_{2, 1}, \dots, A_{2, c_2},
	A_{3, 1}, \dots, A_{3, c_3}) \]
	such that:
	\begin{itemize}
		\item 
			let $L_i$ be the polyline $P_i A_{i, 1} \dots A_{i, c_i} P_{i+1}$
			for each $1 \leq i \leq 3$;
		\item then each polyline $L_i$ is inside triangle $F P_i P_{i+1}$;
		\item the concatenation $L_1 L_2 L_3$ forms a closed non-self-intersecting polyline
			(simple polygon);
		\item the area inside the polygon $L_1 L_2 L_3$ can be triangulated in such a way that
			no diagonal have both endpoints lie on the same polyline $L_i$.
	\end{itemize}
\end{theorem}

So for example, the triangulation mentioned cannot use diagonal $P_1 A_{1, 2}$,
nor $A_{1, 1} A_{1, 3}$.
But it can use diagonal $A_{1, 1} A_{2, 1}$.

Note that of course the edges of polygon $L_1 L_2 L_3$ have both endpoints lie
on the same polyline. We only require the condition for the non-edge diagonal used
in the triangulation.

Equivalently, the requirement is that all non-edge diagonals intersect one of the segments
$F X$, $F Y$, or $F Z$.

To prove the theorem, we need two auxiliary propositions.

\begin{proposition}
	\label{prop:triangulate_three_easy_case}
	Assumptions as in \Cref{theorem:triangulate_between_three_polylines}.
	If there is a point $G \in A$ in triangle $FYZ$,
	segment $XG$ has an intersection $K$ with segment $FY$,
	triangle $FXK$ has no point in $A$,
	then the conclusion of \Cref{theorem:triangulate_between_three_polylines} holds.
\end{proposition}

See \Cref{fig:triangulate_triangle_easy_case} for an illustration.
We require the triangle colored green to be empty.

\begin{proof}
	Extend line $XK$ intersect segment $YZ$ at $H$.
	Apply \Cref{theorem:triangulate_between_two_polylines_finite} twice on
	\begin{itemize}
		\item Pivot point $Y$, separating ray $YF$, point set $(A \cap XYH) \cup \{ X, G \}$
			(therefore the lower cap is $XG$),
		\item Pivot point $Z$, separating ray $ZF$, point set $(A \cap XZH) \cap \{ X, G \}$
			(therefore the lower cap is also $XG$).
	\end{itemize}
	Let $L_1$ be the polyline formed by the first application,
	$L_2$ be the concatenation of the polyline from $Y$ to $G$ and the polyline from $G$ to $Z$,
	$L_3$ be the polyline formed by the second application.

	Since there is no point in $A \cap XFK$, all points in polyline $L_3$ is in triangle $FZX$.
	Since the polygon $L_1 L_2 L_3$ can be divided into two halves by segment $XG$,
	each can be triangulated with only segments intersecting $FY$ or $FZ$,
	we get the conclusion.
\end{proof}

\begin{figure}
    \centering
    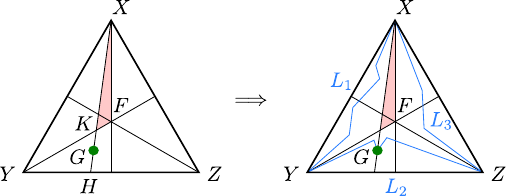
	\caption{Illustration for \Cref{prop:triangulate_three_easy_case}.}
    \label{fig:triangulate_triangle_easy_case}
\end{figure}

\begin{proposition}
	\label{prop:triangulate_three_triangle_case}
	Assumptions as in \Cref{theorem:triangulate_between_three_polylines}.
	If there are points $G$, $H$, $K$ in $A$,
	points $S \in XY$, $R \in YZ$, $T \in ZX$,
	such that $H \in FXY$, $G \in FYZ$, $K \in FZX$,
	polygon $XTKHS$, $YRGHS$, $ZRGKT$ convex,
	$A \cap GHK$ empty,
	$A \cap XTKHS \cap FYZ$ empty,
	$A \cap YRGHS \cap FZX$ empty,
	$A \cap ZRGKT \cap FXY$ empty,
	no point in $A$ is on segment $GR$, $HS$, or $KT$,
	then the conclusion of \Cref{theorem:triangulate_between_three_polylines} holds.
\end{proposition}

See \Cref{fig:triangulate_three_triangle_case} for an illustration.
The assumptions state that the white triangle contains no point in $A$,
each of the three colored polygon is convex,
and if the situation is like the right panel in \Cref{fig:triangulate_three_triangle_case},
the fuchsia region has no point in $A$.

Note that we allow one of the convex polygons above to have an interior angle exactly equal to $180^\circ$.

\begin{figure}
    \centering
    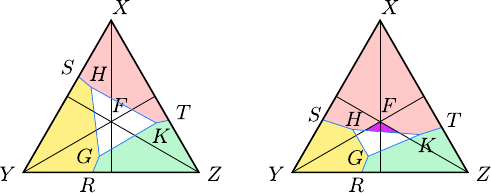
	\caption{Illustration for \Cref{prop:triangulate_three_triangle_case}.}
    \label{fig:triangulate_three_triangle_case}
\end{figure}

\begin{proof}
	Apply \Cref{theorem:triangulate_between_two_polylines_finite} three times on
	\begin{itemize}
		\item pivot point $X$, separating ray $XF$, point set $A \cap XTKHS$;
		\item pivot point $Y$, separating ray $YF$, point set $A \cap YRGHS$;
		\item pivot point $Z$, separating ray $ZF$, point set $A \cap ZRGKT$.
	\end{itemize}
	Then connect the resulting polylines together and merge the resulting triangulations together.
\end{proof}

Now we can prove \Cref{theorem:triangulate_between_three_polylines}.

\begin{proof}
	Divide $A$ into $3$ disjoint sets $A = A_1 \cup A_2 \cup A_3$,
	where $A_i$ is the points in $A$ that is in triangle $F P_i P_{i+1}$.
	Let $M_i$ be the polyline that is the convex hull of
	$\{P_i, P_{i+1}\} \cup A_i$ minus edge $P_i P_{i+1}$.

	Then the concatenation of polylines $M_1 M_2 M_3$ is closed non-self-intersecting,
	therefore it bounds a simple polygon.
	This simple polygon has $c_1+c_2+c_3+3$ edges,
	thus any triangulation of it uses $c_1+c_2+c_3+1$ triangles.

	From \cite{Meisters1975}, the polygon can be triangulated
	with at least $2$ ears (here we define an ear to be a triangle in a triangulation
	that has two edges being edge of the polygon being triangulated).
	Because the region outside each $M_i$ is convex,
	an ear of the triangulation must contain one of the vertices $P_i$,
	therefore there can be at most $3$ ears.

	Recall there are $c_1+c_2+c_3+3$ edges,
	there cannot be any triangle in the triangulation that uses $3$ edges
	of the polygon
	since $c_i>0$ for all $i$,
	there is either $2$ or $3$ ears,
	therefore the following two cases can happen:
	\begin{itemize}
		\item There are $2$ ears, and each of the remaining triangles have exactly one edge on the polyline.
		\item There are $3$ ears, and there is exactly one triangle with no edge on the polyline.
	\end{itemize}

	See \Cref{fig:triangulate_three_total} for an illustration.

	\begin{figure}
		\centering
		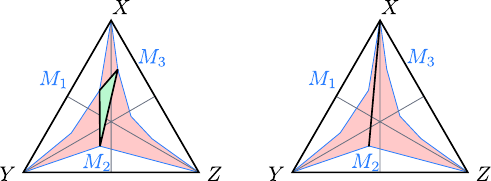
		\caption{Illustration for proof of \Cref{theorem:triangulate_between_three_polylines}.}
		\label{fig:triangulate_three_total}
	\end{figure}

	In the first case, apply \Cref{prop:triangulate_three_easy_case} (illustrated in left panel in \Cref{fig:triangulate_three_total},
	the triangle has the right orientation because of its vertex ordering along polygon $M_1 M_2 M_3$;
	and it contains no point in $A$ because it is contained in that polygon).
	In the second case, apply \Cref{prop:triangulate_three_triangle_case} (illustrated in right panel in \Cref{fig:triangulate_three_total}).

	It remains for us to prove that, in the second case, the assumptions of \Cref{prop:triangulate_three_triangle_case} holds.
	Firstly, any potential fuchsia triangle (as in $XTKHS \cap FYZ$, $YRGHS \cap FZX$, or $ZRGKT \cap FXY$) is either nonexistent
	or contained in the red region.
	Secondly, let $G$ be the vertex on $M_2$, $H$ be the vertex on $M_1$, $K$ be the vertex on $M_3$,
	then we have the ray $HG$ has already crossed $FY$ so it will not cross $FY$ again,
	so ray $HG$ intersects triangle $XYZ$ either in segment $YZ$ or in segment $ZX$;
	if we gradually rotate a ray rooted at $G$ in opposite direction to $GH$ clockwise
	until it points in opposite direction to $GK$,
	there must be some moment that it intersects triangle $XYZ$ on segment $YZ$;
	otherwise ray $KG$ would have intersected triangle $XYZ$ on segment $ZX$,
	which is impossible because segment $KG$ already intersects line $FZ$ once.
\end{proof}

\section{More old proofs}

The following is an old proof of \Cref{prop_triangle_subdivision_3}.
Essentially, it argues that the intersection of the set of points $R$
such that triangle $P_1 P_i Q$ has $c_t$ points
and the set of points $S$ such that triangle $P_1 Q P_{i+1}$ has $c_b$ points is nonempty,
and it does that by traversing the two loci simultaneously.

\begin{proof}[Proof of \Cref{prop_triangle_subdivision_3}]
	Similar to the proof of \Cref{triangle_subdivision},
	for each point $A_j$ in the triangle $P$, define $\sigma(A_j)$ to be the angle $P_i P_1 A_j$,
	then let $j_1, \dots, j_m$ be such that $\sigma(A_{j_1})< \sigma(A_{j_2})<\cdots< \sigma(A_{j_m})$.

	Draw two rays starting at $P_1$ through
	$A_{j_{(c_t+1)}}$ and $A_{j_{(c_t+c_i)}}$, intersecting segment $P_i P_{i+1}$
	at $R$ and $S$ respectively.
	See \Cref{fig:construct_RS} for an illustration.

	\begin{figure}
		\centering
		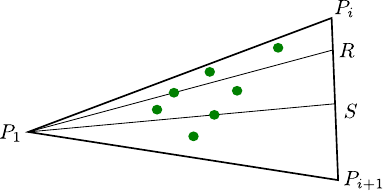
		\caption{Illustration for the construction of $R$ and $S$.}
		\label{fig:construct_RS}
	\end{figure}

	The idea is the following:
	the triangle $P$ has been divided into three parts as required,
	triangle $P_1 P_i R$,
	polygon $P_i R P_1 S P_{i+1}$,
	and triangle $P_1 S P_{i+1}$.
	We call these three parts
	the \emph{top part}, \emph{middle part} and \emph{bottom part} respectively.
	Furthermore, if the points on segment $P_1 R$ and $P_1 S$ are considered to be inside the middle part,
	then the number of points in each part is correct---%
	$c_t$, $c_i$, and $c_b$ respectively.

	The problem is that the middle part,
	polygon $P_i R P_1 S P_{i+1}$,
	is neither a triangle nor is convex.
	Our plan is thus to move $R$ and $S$ gradually toward each other
	while maintaining the same number of points in each part,
	when they coincide then we would have gotten a point $Q$.

	During the process, we assume that any point that lies exactly on the perimeter of the middle part
	(that is, on one of the segments in the polyline $P_i R P_1 S P_{i+1}$)
	belongs to the middle part.

	Now we explain how exactly the movement is performed.
	For point $R$, if it is possible to move away from $P_i$ we prefer to do so,
	otherwise we move towards $P_1$.\footnote{%
		To be clear, ``move $R$ away from $P_i$'' means ``move $R$ along the ray opposite to ray $R P_i$'',
		and ``move $R$ towards $P_1$'' means ``move $R$ along the line segment $R P_1$''.
	}
	Similarly, for point $S$, if it is possible to move away from $P_{i+1}$ we prefer to do so,
	otherwise we move towards $P_1$.
	We control the speed of both points to maintain the invariant that $R$ and $S$ has the same distance to line $P_i P_{i+1}$. (This is equivalent to the statement that $RS$ is parallel to $P_i P_{i+1}$, except when $R$ and $S$ coincide.)

	At the very beginning, there is a point on segment $P_1 R$ and $P_1 S$, so they will start off moving towards $P_1$.
	During the process, the following types of events can happen.

	\begin{enumerate}
		\item Point $R$ is moving towards $P_1$ (because there exists a point $A$ on segment $R P_1$), but segment $R P_i$ hits a new point $B$.  See \Cref{fig:movement_1} for an illustration.

			In this case, at the point where the new point would have passed through $R P_i$ to the middle part,
			we simultaneously redirect $R$ to stop moving towards $P_1$ and start moving away from $P_i$.

			\begin{figure}
				\centering
				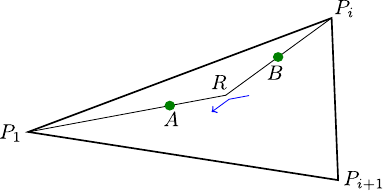
				\caption{One possible kind of event that can happen during the movement process.}
				\label{fig:movement_1}
			\end{figure}
		\item Point $R$ is moving towards $P_1$, and $R$ hits a new point.

			This case is very similar to the case above, and the handling method is the same:
			redirect $R$ to start moving away from $P_i$.
		\item Point $R$ is moving away from $P_i$, but segment $P_1 R$ hits a new point $A$.
			See \Cref{fig:movement_2} for an illustration.

			If it were to continue moving past that point, $A$ would pass from the middle part to the top part.
			So, at this point, we redirect $R$ to start moving towards $P_1$ instead.

			\begin{figure}
				\centering
				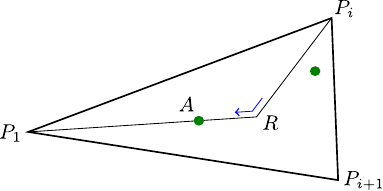
				\caption{Another possible kind of event that can happen.}
				\label{fig:movement_2}
			\end{figure}
	\end{enumerate}

	There are three ways this procedure could end.
	See \Cref{fig:movement_end} for an illustration.
	\begin{enumerate}
		\item The two points $R$ and $S$ are already coinciding, moving to the left (towards $P_1$) together, there is some point $A$ on segment $P_1 R$
			that should be counted towards the middle part, the two points $R$ and $S$ are moving towards $P_1$,
			and segment $R P_i$ hits another point $B$.

			In this case, we perform the swap---%
			count $A$ towards the top part, count $B$ towards the middle part---%
			then end the procedure.

			There is also a symmetric variant (vertically flip the diagram) where the point $B$ is found on segment $S P_{i+1}$ instead.

			(In the illustration, we draw a little gap between $R$ and $S$---%
			this is to illustrate that pont $A$ is counted towards the middle part)

		\item As above, but segment $R P_i$ hits the point $A \in P_1 R$.

			In this case we merely count $A$ towards the middle part (as it has always been)
			then end the procedure.

		\item The two points $R$ and $S$ are both moving away from $P_i$ and $P_{i+1}$ respectively, and they coincide.

			As above, ending the procedure works.
	\end{enumerate}

	\begin{figure}
		\centering
		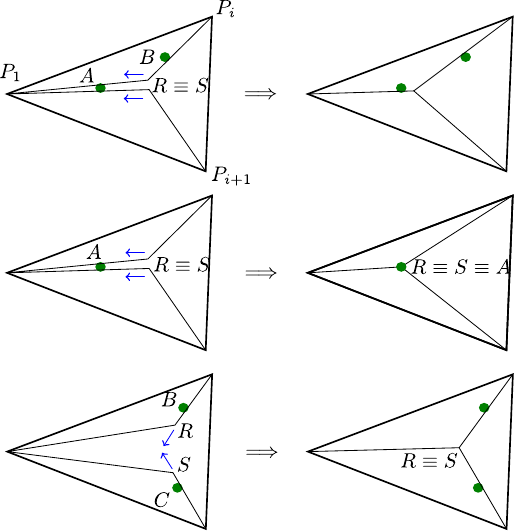
		\caption{Illustration for the possible ways the movement could end.}
		\label{fig:movement_end}
	\end{figure}

	Since no three lines are concurrent,
	in the first case there is no point on segment $R P_{i+1}$,
	and in the third case there is no point on segment $P_1 R$.
	As such, we can pick $Q$ very close to $R$ or $S$ in such a way that the points on the boundary between parts
	are counted towards the correct part.
\end{proof}

\begin{theorem}
	Let $P_0$ be the set of points in the double circle.
	Let $I' \subseteq P_0 \setminus \CH(P_0)$ be arbitrary.
	Define $P = \CH(P_0) \cup I'$.
	Then $P$ is universal.
\end{theorem}
\begin{proof}
	Let $\CH(P) = \{ P_1, \dots, P_n \}$
	and $I' = \{A_1, \dots, A_m\}$
	where $0 \leq m \leq n$,

	Then let $Q = \{ Q_1, \dots, Q_n, B_1, \dots, B_m \}$ be an arbitrary set of points.
	Assume the prescribed cyclic mapping sends $P_i$ to $Q_i$.

	For each $1 \leq i \leq m$ where the edge $P_i P_{i+1}$ has the interior point \emph{deleted},
	add a point $B^*_i$ near the midpoint of $Q_i Q_{i+1}$, shifted slightly inwards,
	in such a way that the shift is so small that, and the point set $Q \cup \{ B^*_i \}$ is still
	in general position.

	Then $Q \cup \{ B^*_i \}$ has exactly $2n$ points. Use
	\Cref{theorem:universality_double_circle},
	there exists a reordering $\{ B'_1, \dots, B'_n \}$ of $\{ B_1, \dots, B_m \} \cup \{ B^*_i \}$
	such that the triangles $Q_i B'_i Q_{i+1}$ has no intersection
	except at vertices $Q_i$ of $\CH(Q)$.

	From our construction of $B^*_i$, we must have $B^*_i = B'_i$ for every $i$ such that
	$B^*_i$ is constructed.

	Construct the mapping $f$ accordingly, then get any triangulation of the polygon formed by points of $Q$,
	and use \Cref{prop:porting_triangulation} to port the triangulation to $P$
	similar to the proof of \Cref{theorem:universality_double_circle}.
	We're done.
\end{proof}

When we pass from the double circle to the generalized double circle, it is not always true that a triangulation of the polygon formed by $Q$ can be ported to a triangulation of the polygon formed by the unavoidable spanning cycle of the generalized double circle.

Nonetheless, we still have the following:

\begin{theorem}
	\label{theorem:result_general_double_circle_1}
	Let $c_1$, $c_2$, $c_3$ be positive integers.
	Let $P$ be the $(c_1, c_2, c_3)$-generalized double circle.
	Then $\mathcal T(P, Q, f_0) \leq 1$
	for all suitable point sets $Q$ in general position and cyclic mappings $f_0$,
	as long as no three lines over $Q$ are concurrent.
\end{theorem}

Note that unlike \Cref{theorem:existing_result_Steiner_point},
the bound is independent of the number of interior points.

We will be able to prove even a stronger theorem,
but since the proof of this is significantly simpler,
we present it first for the sake of exposition.

\begin{proof}
	Apply \Cref{triangle_subdivision} on point set $Q$ with parameters $(c_1, c_2, c_3)$
	to find a point $B$ inside $Q$, such that the parts $Q_1 B Q_2$, $Q_2 B Q_3$, $Q_3 B Q_1$
	has $c_1$, $c_2$ and $c_3$ points in $Q \setminus \CH(Q)$ respectively.

	Let $A$ be the center of the triangle $\CH(P)$.

	Then, let the Steiner point be $A$ and $B$ for point set $P$ and $Q$ respectively.
	In each part of $Q$, order the points in the interior by its direction from $B$.
	Then use that ordering to find a correspondence between points in $P$ and points in $Q$.

	Find any triangulation of the polygon on the outside of $Q$,
	then port it to $P$. We're done.
\end{proof}

\begin{theorem}
	\label{theorem:result_general_double_circle_2}
	Let $c_1$, $c_2$, $c_3$, $P$ be as above.
	Then $P$ is universal.
\end{theorem}

\begin{proof}
	{
		\hfuzz=4pt
		Let 
		the points of the $(c_1, c_2, c_3)$-generalized double circle $P$
		be 
		$\{P_1, \dots, P_n, \allowbreak
			A_{1, 1}, \dots, A_{1, c_1}, \allowbreak
			A_{2, 1}, \dots, A_{2, c_2}, \allowbreak
		A_{3, 1}, \dots, A_{3, c_3}\}$,
		notation as in \Cref{def_general_double_circle}.

	}

	Given any point set $Q = \{ Q_1, \dots, Q_n, B_1, \dots, B_m \}$
	in general position such that $|P|=|Q|$, $\absCH{P}=\absCH{Q}$.

	Since the existence of a triangulation only depends on the order type of $Q$,
	we may perturb $Q$ slightly so that no three lines are concurrent.

	Then apply \Cref{theorem:result_general_double_circle_1} to get a point $F \in \CH(Q)$ as above.

	See the left panel in \Cref{fig:triangle_divided_into_parts} for an illustration.

	\begin{figure}
		\centering
		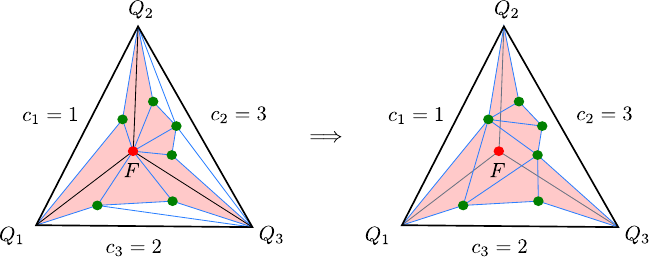
		\caption{Illustration for proof of \Cref{theorem:result_general_double_circle_2}.}
		\label{fig:triangle_divided_into_parts}
	\end{figure}

	By construction, the triangle $Q_i F Q_{i+1}$ contains $c_i$ points in $\{B_1, \dots, B_m \}$
	for each $1 \leq i \leq 3$.

	The plan is the following.
	We want to find a bijection
	$f \colon
	\{ A_{1, 1}, \dots, A_{1, c_1}, \allowbreak
		A_{2, 1}, \dots, A_{2, c_2}, \allowbreak
	A_{3, 1}, \dots, A_{3, c_3}\}
	\to \{ B_1, \dots, B_m \}$
	such that the polygon \[
		Q_1
		f(A_{1, 1}) f(A_{1, 2}) \dots f(A_{1, c_1})
		Q_2
		f(A_{2, 1})  \dots f(A_{2, c_2})
		Q_3
		f(A_{3, 1})  \dots f(A_{3, c_3})
	\]
	is not self-intersecting, and the region inside it can be triangulated
	without using the Steiner point $F$,
	and each segment used in the triangulation of this polygon
	must intersect at least one of the segments $Q_1 F$, $Q_2 F$, $Q_3 F$.
	
	See \Cref{fig:triangle_divided_into_parts} for an illustration.

	To do that, we just apply \Cref{theorem:triangulate_between_three_polylines} and we're done.
\end{proof}

\begin{remark}
	If we fix the bijection $f$ in the proof above, there might not be a satisfying triangulation.
	See \Cref{fig:counterexample_no_inner_triangulation} for an example.
	On the triangle on the left, any triangulation of the polygon marked in red
	must contain the three blue segments.
	On the triangle on the right, we modify the bijection,
	and there exists a triangulation of the polygon marked in red
	that can be extended to a compatible triangulation
	with the $(2, 2, 2)$-generalized double circle.

	\begin{figure}
		\centering
		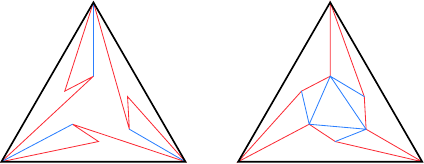
		\caption{Example of a bijection $f$ that gives no compatible triangulation.}
		\label{fig:counterexample_no_inner_triangulation}
	\end{figure}
\end{remark}

\subsection{Alternative Proof of \Cref{theorem:triangulate_between_three_polylines}}

This is my previous proof, which is less neat.

	We only need to consider the case where \Cref{prop:triangulate_three_easy_case} does not happen.

\begin{figure}
    \centering
    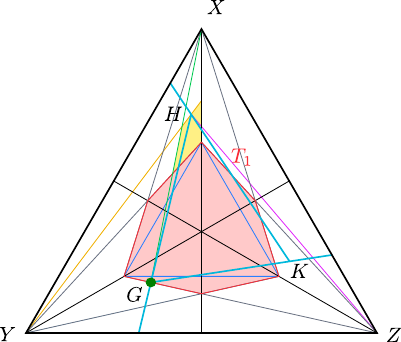
    \caption{Caption}
    \label{fig:triangulate_three_hard_case}
\end{figure}

	Draw concentric equilateral triangles around $F$ that is a shrunken version of $XYZ$.
	(One of those is drawn in blue in \Cref{fig:triangulate_three_hard_case}.)
	Then draw line segments connecting the vertices of triangle $XYZ$ to the vertices of that triangle,
	let $T_1$ be the hexagon enclosed by these segments. (Colored red in \Cref{fig:triangulate_three_hard_case}.)

	Pick $T_1$ so that it is the smallest such hexagon containing a point in $A$,
	and let $G \in A$ be the point in $T_1$.
	(Thus $G$ is on the boundary of $T_1$.)

	By construction of $T_1$, the interior of $T_1$ has no point in $A$.

	Without loss of generality assume $G$ is in triangle $FYZ$, and $G$ and $Y$ are on the same side of line
	$FX$.

	Since \Cref{prop:triangulate_three_easy_case} does not happen, there is some
	point $H \in A$
	that is in both triangle $FXG$ and triangle $FXY$.
	Again, draw concentric hexagons as above, and let $T_2$ be the smallest hexagon
	that contains such a point $H$.
	(In \Cref{fig:triangulate_three_hard_case}, point $H$ is a corner of the cyan triangle; or alternatively the intersection of the yellow segment and the fuchsia segment.)

	By construction of $T_2$, part of the interior of $T_2$ (colored yellow in \Cref{fig:triangulate_three_hard_case})
	has no point in $A$.
	Because there is no point of $A$ in the interior of $T_1$, $H$ and $X$ must be on the same side of line $FZ$.

	Again, since \Cref{prop:triangulate_three_easy_case} does not happen, there is some
	point $K \in A$ that is in both triangle $FZH$ and triangle $FZX$.
	(In \Cref{fig:triangulate_three_hard_case}, point $K$ is a corner of the cyan triangle.)
	Pick such a point $K$ that is nearest to line $GH$.
	Then there is usually no other point in $A$ in triangle $GHK$.

	Extend rays $KH$, $HG$, $GK$. By \Cref{prop:intersecting_edge} and similar arguments, they must intersect
	segment $XY$, $YZ$, $ZX$ respectively.
	These rays divide the set of points in $A$ into three convex regions,
	apply \Cref{prop:triangulate_three_triangle_case}.

\begin{figure}
	\centering
	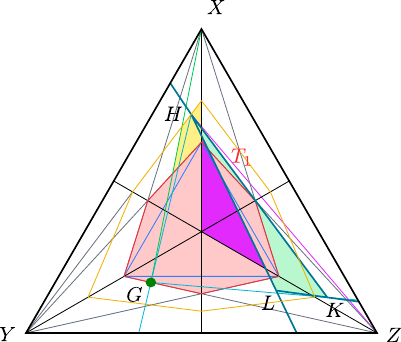
	\caption{Caption}
	\label{fig:triangulate_three_hard_case_2}
\end{figure}

	There is a special case where there \emph{is} another point in $A$ in triangle $GHK$.
	Let $L$ be one such point such that angle $KHL$ is minimum.
	(Thus the region colored green in \Cref{fig:triangulate_three_hard_case_2} has no point in $A$.)

	By the choice of $K$, there is no point in $A \cap FZH \cap FZX$ closer to $GH$ than $K$.
	But any point in the interior of triangle $GHK$ must be closer to $GH$ than $K$.
	Therefore $L \in GHK \setminus (FZH \cap FZX)$.

	Point $L$ must not be in triangle $FZX$,\footnote{
		Let $M$ be the intersection of $ZH$ and $ZX$.
		Then $FZH \cap FZX = FZM$, and $L \in GHK \setminus FZM \subseteq GHZ \setminus FZM = GHMFZ$.
		This polygon has no intersection with $FZX$.
	}
	nor in triangle $FXY$,\footnote{
		Let $M$ be as above.
		We have $GHK \cap FXY \subseteq GHZ \cap FXY \subseteq GHMF$.
		This is contained in the red and yellow regions.
	}
	thus it is in triangle $FYZ$.
	Furthermore $G$ and $L$ are on different sides of line $FX$.\footnote{Otherwise $L$ would be in the red region.}
	See \Cref{fig:triangulate_three_hard_case_2} for an illustration.

	Then we can replace $G$ with $L$, noticing triangle $KHL$ has no point in $A$,\footnote{
		Triangle $KHL$ is fully covered by the red, yellow and green colored regions.
	}
	and the fuchsia region in \Cref{fig:triangulate_three_hard_case_2} has no point in $A$,\footnote{
		By the same argument that proves $L$ cannot be in triangle $FZX$.
	}
	then extend rays $KH$, $HL$, $LK$ and apply \Cref{prop:triangulate_three_triangle_case} exactly as above.

	There is an alternative case, see \Cref{fig:triangulate_three_hard_case_3} for an illustration.
	This case poses no difficulty, the triangle colored fuchsia near $F$ is included in $T_1$
	so it has no point in $A$.

\begin{figure}
    \centering
    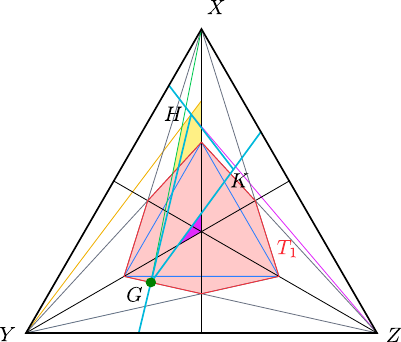
    \caption{Caption}
    \label{fig:triangulate_three_hard_case_3}
\end{figure}

Now we fill in some missing details in the proof.

\begin{proposition}
	\label{prop:intersecting_edge}
	Let $G$ be in triangle $FYZ$ such that $G$ and $Y$ are on the same side of line $FX$.
	Let $K$ be in triangle $FZX$.
	Then ray $GK$ intersects triangle $XYZ$ on segment $XZ$.
\end{proposition}

See \Cref{fig:triangulate_three_intersection_helper} for an illustration.
Point $G$ is restricted to the yellow region, and point $K$ is restricted to the blue region.

\begin{figure}
    \centering
    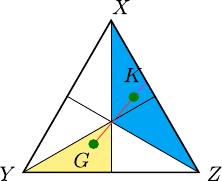
	\caption{Illustration for \Cref{prop:intersecting_edge}.}
    \label{fig:triangulate_three_intersection_helper}
\end{figure}
\begin{proof}
	Since $G$ and $K$ are on different sides of line $FX$,
	and two lines intersect at most once,
	the ray $GK$ cannot intersect line $FX$ again.
	Similarly, ray $GK$ cannot intersect line $FY$ again.
	Therefore the intersection point must be in the region $FXZ$.
\end{proof}

\subsection{Failed Attempts at Proof of \Cref{theorem:triangulate_between_three_polylines}}

It is unfortunate that the proof of \Cref{theorem:triangulate_between_three_polylines}
is not very symmetric.
The following are some alternative attempts.

We use a projective transformation to transform $\{ X, Y, Z, F \}$ to three vertices
and the center of a equilateral triangle.
As such we only need to consider the case where $XYZ$ is an equilateral triangle
and $F$ is its center.

	We only need to consider the case where \Cref{prop:triangulate_three_easy_case} does not happen.

Consider the set of points of $A$ that is on the same side as $Y$ of line $FX$
(there must be some point because $c_1>0$).
Within those, pick a point $G$ such that angle $FXG$ is minimized.
Because \Cref{prop:triangulate_three_easy_case} does not happen, $G$ is in triangle $FXY$.

As such, the part of the plane spanned by angle $FXG$ contains no point in $A$.
See \Cref{fig:triangulate_three_better_illustration} for an illustration---%
the green region contains no point in $A$.

\begin{figure}
    \centering
    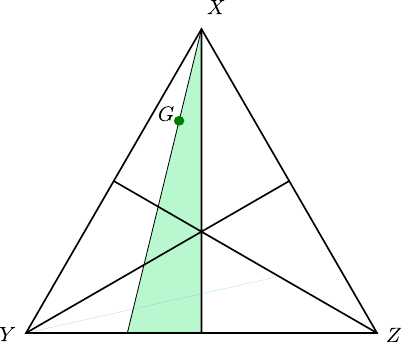
    \caption{Caption}
    \label{fig:triangulate_three_better_illustration}
\end{figure}

Perform the same procedure for the remaining vertices and half planes to get $6$ rays,
with $2$ starting from each vertex.
See \Cref{fig:triangulate_three_better_illustration_2} for an illustration,
the green region is bounded by the $6$ rays obtained above.

\begin{figure}
    \centering
    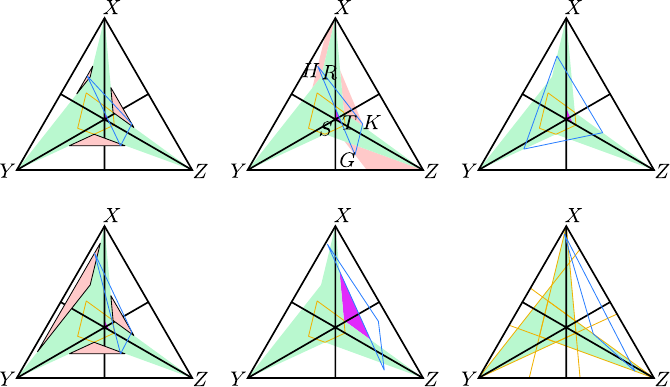
    \caption{Caption}
    \label{fig:triangulate_three_better_illustration_2}
\end{figure}

Note that we discard from the green region the part of the ray that goes through the common region at the center
(colored in yellow).

The remaining part of the triangle after deleting the green region
forms three triangles (white in \Cref{fig:triangulate_three_better_illustration_2}).
Let $R$, $S$ and $T$ be the tip of these triangles.

In triangles $XYR$, $YZS$, $ZXT$, there are $c_1$, $c_2$ and $c_3$ points in $A$ respectively.
This is because there's no point in $A$ in the green region.

Now, the problematic part.
We wish to pick $H \in XYR$, $G \in YZS$, $K \in ZXT$
in some way such that
\Cref{prop:triangulate_three_triangle_case} can be applied on triangle $GHK$.
For that to happen, we need triangle $GHK$ to not contain any point in $A$,
have the correct orientation
(thus not like the bottom right panel in \Cref{fig:triangulate_three_better_illustration_2}),
and any possible fuchsia triangle has no point.

In the illustrations in \Cref{fig:triangulate_three_better_illustration_2}, triangle $GHK$ is colored in blue.

\begin{itemize}
	\item First attempt: pick $G$, $H$ and $K$ to maximize the distance to the boundary of triangle $XYZ$.

		By the construction, the red regions in top left panel in \Cref{fig:triangulate_three_better_illustration_2}
		contains no point in $A$,
		and their outer side contains at least one point in $A$.

		Problem: triangle $GHK$ might still contain some point in $A$.

	\item Second attempt: pick $G$, $H$ and $K$ such that the area of triangle $GHK$ is minimum.
		By construction, there is no point in $A$ in triangle $FXY$ that has distance to line $GK$ less than that of $H$, for example.

		The problem: we cannot ensure the orientation is correct.
		See bottom right panel
		in \Cref{fig:triangulate_three_better_illustration_2}.

		And if we force the orientation to be correct, the situation may be in bottom middle panel
		in \Cref{fig:triangulate_three_better_illustration_2}.
		The fuchsia triangle cannot be guaranteed to have no point in $A$.

	\item Third attempt: pick points on the boundary of the green region.

		Problem: see top right panel
		in \Cref{fig:triangulate_three_better_illustration_2}.
\end{itemize}

\end{document}